\documentclass[]{amsart}
\usepackage[utf8]{inputenc}
\usepackage{graphicx}

\DeclareUnicodeCharacter{2212}{-}

\usepackage{amsmath, amssymb}
\usepackage{amsthm}
\usepackage{enumerate}
\usepackage{mathtools}
\usepackage{tikz-cd}
\usepackage{comment}
\usepackage{color}
\usepackage{hyperref}

\newcommand{\ZZZ}{\operatorname{Z}}

\newcommand{\GL}{\operatorname{GL}}

\newcommand{\NS}{\operatorname{NS}}

\newcommand{\p}{\operatorname{\mathbb{P}}}
\newcommand{\C}{\operatorname{\mathbb{C}}}
\newcommand{\Z}{\operatorname{\mathbb{Z}}}
\newcommand{\h}{\operatorname{\mathbb{H}^{\infty}}}

\newcommand{\s}{\mathcal{S}}

\newcommand{\Ind}{\operatorname{Ind}}
\newcommand{\A}{\operatorname{\mathbb{A}}}
\newcommand{\Q}{\operatorname{\mathbb{Q}}}
\newcommand{\R}{\operatorname{\mathbb{R}}}
\newcommand{\F}{\operatorname{\mathbb{F}}}
\newcommand{\kk}{{\bf k}}

\newcommand{\ZZ}{\mathcal{Z}}
\newcommand{\id}{\operatorname{id}}

\newcommand{\car}{\operatorname{char}}

\newcommand{\norm}{\operatorname{Norm}}

\newcommand{\aut}{\operatorname{Aut}}
\newcommand{\edo}{\operatorname{End}}

\newcommand{\tr}{\operatorname{tr}}
\newcommand{\Bir}{\operatorname{Bir}}
\newcommand{\Cr}{\operatorname{Cr}}
\newcommand{\PGL}{\operatorname{PGL}}
\newcommand{\SL}{\operatorname{SL}}

\newcommand{\PSL}{\operatorname{PSL}}
\newcommand{\J}{\mathcal{J}}

\newcommand{\B}{\mathcal{B}}

\newcommand{\arccosh}{\operatorname{arccosh}}

\def\dashmapsto{\mapstochar\dashrightarrow}
\newcommand{\Ax}{\operatorname{Ax}}

\def\dashmapsto{\mapstochar\dashrightarrow}

\usepackage[shortlabels]{enumitem}
\setlist[enumerate]{label=\rm{(\arabic*)}}
\setlist[enumerate,2]{label=\rm({\it\roman*})}
\setlist[itemize]{label=\raisebox{0.25ex}{\tiny$\bullet$}}

\theoremstyle{plain}
\newtheorem{theorem}{Theorem}[section]
\newtheorem{lemma}[theorem]{Lemma}
\newtheorem{proposition}[theorem]{Proposition}
\newtheorem{corollary}[theorem]{Corollary}

\theoremstyle{definition}
\newtheorem{definition}{Definition}[section]

\newtheorem{example}[theorem]{Example}

\newtheorem{remark}{Remark}[section]

\newtheorem{question}{Question}[section]
\begin{document}


\address{Office 682, Huxley Building\\
	Mathematics Department\\
	Imperial College London\\
	180 Queen's Gate\\
	London SW7 2AZ, UK }
\email{christian.urech@gmail.com}
\thanks{The author gratefully acknowledges support by the Swiss National Science Foundation Grant "Birational Geometry"  PP00P2 128422 /1 as well as by the Geldner-Stiftung, the FAG Basel, the Janggen P\"ohn-Stiftung and the State Secretariat for Education,
	Research and Innovation of Switzerland}

\author{Christian Urech}
\title{Simple groups of birational transformations in dimension two}
\date{February 2018}

\maketitle
\begin{abstract}
	We classify simple groups that act by birational transformations on compact complex K\"ahler surfaces. Moreover, we show that every finitely generated simple group that acts non-trivially by birational transformations on a projective surface over an arbitrary field is finite.
\end{abstract}
\tableofcontents

\section{Introduction and results}

Let $S$ be a surface over a field $\kk$ and denote by $\Bir(S)$ its group of birational transformations. If $S$ is rational, this group is particularly rich and interesting. In this case it is isomorphic to the {\it plane Cremona group}  
\[
\Cr_2(\kk)\coloneqq \Bir(\p^2_{\kk}). 
\]
In the last decade numerous results about the group structure of the plane Cremona group have been proven (see \cite{cantatrev} for an overview). One of the main techniques to better understand infinite subgroups of $\Cr_2(\kk)$ was the construction of an action by isometries of the plane Cremona group on an infinite dimensional hyperbolic space $\h(\p^2)$ and the use of results from hyperbolic geometry and group theory. For instance, it had been a long-standing open question, whether the plane Cremona group is simple as a group until Cantat and Lamy showed in 2012 that it is not (\cite{cantat2013normal}). The main idea to prove this result was to use techniques from small cancellation theory, an approach that has been refined by Shepherd-Barron and Lonjou (see \cite{Shepherd-Barron:2013qq}, \cite{MR3533276}). In this paper we take these results as a starting point to give a classification of all simple groups that act non-trivially by birational transformations on compact complex K\"ahler surfaces.  Our main result is the following:

\begin{theorem}\label{mainthm}
	Let $G$ be a simple group. Then
	\begin{enumerate}
		\item $G$ acts non-trivially by birational transformations on a rational complex projective surface if and only if $G$ is isomorphic to a subgroup of $\PGL_3(\C)$.
		\item  $G$ acts non-trivially by birational transformations on a non-rational compact complex K\"ahler surface of negative Kodaira dimension if and only if $G$ is finite or isomorphic to a subgroup of $\PGL_2(\C)$. 
		\item $G$ acts non-trivially by birational transformations on a compact complex K\"ahler surface $S$ of non-negative Kodaira dimension if and only if $G$ is finite.
	\end{enumerate}
\end{theorem}

It should be emphasized that part (2) and (3) of Theorem \ref{mainthm} are not hard to prove using some well-known facts about groups of birational transformations of non-rational compact complex surfaces (see Section \ref{nonrational}). Whereas for the proof of the first part we will use some rather difficult theorems and lengthy arguments. 

An element $f\in\Cr_2(\kk)$ is called {\it elliptic, parabolic} or {\it loxodromic} if the isometry of the hyperbolic space $\h(\p^2)$ induced by $f$ is elliptic, parabolic or loxodromic respectively. This distinction is closely linked to the dynamical behavior of $f$ (see Section \ref{degreesandtypes}). The subgroup $\J\subset\Cr_2(\C)$ of elements preserving a given rational fibration is called the {\it de Jonqui\`eres subgroup}. It is isomorphic to $\PGL_2(\C)\ltimes\PGL_2(\C(t))$. A faithful and regular action of an algebraic group $H$ on a rational projective surface $S$ induces an embedding of $H$ into $\Cr_2(\C)$ defined up to conjugation; the image is called an {\it algebraic subgroup} (see Section \ref{zariskitop} for details). We will give the following precise description of simple subgroups of $\Cr_2(\C)$:

\begin{theorem}\label{mainsimple}
	Let $G\subset\Cr_2(\C)$ be a simple group. Then:
	\begin{enumerate}
		\item $G$ does not contain loxodromic elements.
		\item If $G$ contains a parabolic element, then $G$ is conjugate to a subgroup of $\J$.
		\item If all elements in $G$ are elliptic, then either $G$ is a simple subgroup of an algebraic subgroup of $\Cr_2(\C)$, or  $G$ is conjugate to a subgroup of $\J$.
	\end{enumerate}
\end{theorem}

Theorem \ref{mainthm} naturally leads to the question about the classification of simple subgroups of $\PGL_3(\C)$ and $\PGL_2(\C)$. Obvious classes of simple subgroups of $\PGL_2(\C)$ are finite simple subgroups, or subgroups of the form $\PSL_2(\kk)$, where $\kk\subset\C$ is a subfield. It is unclear, whether there exist other examples. In fact, already the following question seems to be open:

\begin{question}\label{pslquestion}
	Does $\PSL_2(\Q)$ contain proper infinite simple subgroups?
\end{question}	
	
Question \ref{pslquestion} can be seen in the context of a more general question that has been asked by McKay and Serre  (see \cite{cantatahs}, \cite[15.57.]{Mazurov:2014aa} for details). 

 If we consider only finitely generated simple subgroups of $\Bir(S)$, we do not have to restrict ourselves to the field of complex numbers and we can use different techniques.
Recall that a group $G$ satisfies the {\it property of Malcev} if every finitely generated subgroup $\Gamma\subset G$ is {\it residually finite}, i.e.\,for every element $g\in\Gamma$ there exists a finite group $H$ and a homomorphism $\varphi\colon\Gamma\to\ H$ such that $g$ is not contained in the kernel of $\varphi$. Malcev showed that linear groups satisfy this property (\cite{malcev1940isomorphic}). Other groups that fulfill the property of Malcev include automorphism groups of schemes over any commutative ring (\cite[Corollary 1.2]{MR693651}). In \cite{MR2811600}, Cantat asked whether the plane Cremona group has the property of Malcev, a question that is still open. 
Finitely generated simple subgroups of groups with the property of Malcev are always finite. We will prove the following result (where surfaces are always considered to be geometrically irreducible): 

\begin{theorem}\label{fgsimple}
 Let $S$ be a surface over a field $\kk$ and $\Gamma\subset\Bir(S)$ a finitely generated simple group. Then $\Gamma$ is finite.
\end{theorem}

In other words, all finitely generated simple groups of birational transformations in dimension 2 are finite. From the classification of finite subgroups of $\Cr_2(\C)$ (see \cite{MR2641179}) we obtain in particular:

\begin{corollary}\label{fingenclass}
	A finitely generated simple subgroup of $\Cr_2(\C)$ is isomorphic to 
	\[
	\Z/p\Z, \text{ for some prime $p$},\quad \mathcal{A}_5,\quad \mathcal{A}_6, \quad\PSL_2(7).
	\]
\end{corollary}
The conjugacy classes of these finite groups are classified in \cite{MR2641179}.

\subsection{Acknowledgements} I express my warmest thanks to my PhD-advisors J\'e\-r\'e\-my Blanc and Serge Cantat  for their guidance during this work, their constant support and helpful comments on previous versions of this text. I am indebted to Michel Brion and Ivan Cheltsov for numerous helpful comments on this work. I also thank Vincent Guirardel, St\'ephane Lamy, Anne Lonjou and Susanna Zimmermann for interesting discussions.

\section{Preliminaries}
We always assume that $\kk$ is a fixed algebraically closed field (unless stated explicitely otherwise). If we choose homogeneous coordinates $[x:y:z]$ of $\p_{\kk}^2$, every element $f\in\Cr_2(\kk)$ is given by
\[
[x:y:z]\dashmapsto [f_0(x,y,z):f_1(x,y,z):f_2(x,y,z)],
\]
where $f_0, f_1, f_2\in \kk[x,y,z]$ are homogeneous polynomials of the same degree and without a non-constant common factor. We will identify $f$ with $[f_0:f_1:f_2]$ by abuse of notation. With respect to affine coordinates $(x,y)=[x:y:1]$, the birational transformation $f$ is given by $(x,y)\dashmapsto (F, G)$, where $F=\frac{f_0(x,y,1)}{f_2(x,y,1)}, G=\frac{f_1(x,y,1)}{f_2(x,y,1)}\in\kk(x,y)$. When working with affine coordinates, we identify $f$ with $(F,G)$

\subsection{The Picard-Manin space}\label{bubblesection}
Let $X$ be a projective surface. Then $\Bir(X)$ acts by isometries on an infinite dimensional hyperbolic space $\h(X)$. Since this construction has been described in detail in various places, we will just briefly sketch the main ideas and refer to \cite{MR833513} and \cite{MR2811600} for proofs (see also \cite{cantatrev} and \cite{cantat2013normal}).

We start by a construction that is due to Manin (\cite{MR833513}).  The {\it bubble space} $\B(X)$ of a smooth projective surface $X$ is the set of all points that belong to $X$ or are infinitely near to $X$. It is defined as the set of all triples $(y,Y,\pi)$, where $Y$ is a smooth projective surface, $y\in Y$ and $\pi\colon Y\to X$ a birational morphism, modulo the following equivalence relation: A triple $(y, Y, \pi)$ is equivalent to $(y', Y', \pi')$ if there exists a birational map $\varphi\colon Y\dashrightarrow Y'$ that restricts to an isomorphism in a neighborhood of $y$ that maps $y$ to $y'$, and that satifies $\pi'\circ\varphi=\pi$. A {\it proper point}  of $X$ is a point $p\in\B(X)$ that is equivalent to $(x,X,\id)$. All points in $\B(X)$ that are not proper are called {\it infinitely near}. If there is no ambiguity, we will denote a point  $(y, Y, \pi)$ in the bubble space just by $y$.

Denote by $\B(f)$  the base-points of a birational map $f\colon X\dashrightarrow Y$ of projective surfaces $X$ and $Y$. A birational morphism $\pi\colon X\to Y$ of surfaces induces a bijection $(\pi_1)_\bullet\colon\B(X)\to\B(Y)\setminus\B(\pi^{-1})$, where  $\pi_\bullet(x,X,\varphi)\coloneqq(x,X, \pi\circ\varphi)$. A birational transformation of smooth projective surfaces $f\colon X\dashrightarrow Y$ defines a bijection  $f_\bullet\colon\B(X)\setminus\B(f)\to\B(Y)\setminus \B(f^{-1})$ by  $f_\bullet\coloneqq (\pi_2)_\bullet\circ(\pi_1)_\bullet^{-1}$, where $\pi_1\colon Z\to X$, $\pi_2\colon Z\to Y$ is a minimal resolution of $f$. 

We define the {\it Picard-Manin space} of a smooth projective surface $X$ by
\[
\ZZ(X)\coloneqq\lim_{\pi\colon Y\to X}\NS(X),
\] 
where the direct limit is taken over all birational morphisms of smooth projective surfaces $\pi\colon Y\to X$.  The intersection forms on the groups $\NS(Y)$ induce a quadratic form on $\ZZ(X)$ of signature $(1,\infty)$, by the Hodge index Theorem. 
For a point $p\in\B(X)$ we denote by $e_p$ the class of the exceptional divisor of the blow-up of $p$ in $\ZZ(X)$. The Picard-Manin space has the following decomposition
\[
\ZZ(X)=\NS(X)\oplus\bigoplus_{p\in\B(X)}\Z e_p,
\]
where $e_p\cdot e_p=-1$ and $e_p\cdot e_q=0$ for all $p\neq q$, as well as $e_p\cdot D=0$ for all $D\in\NS(X)$. 
 Consider the following completion of the real vector space $\ZZ(X)\otimes\R$:
\[
\ZZZ(X)\coloneqq\{v+\sum_{p\in\B(X)}a_pe_p\mid v\in\NS(X)\otimes\R, a_p\in\R, \sum_{p\in\B(X)}a_p^2<\infty\}.
\] 
The intersection form on $\ZZ(X)\otimes\R$ extends continuously to a quadratic form on $\ZZZ(X)$ with signature $(1,\infty)$. We fix a vector $e_0\in\ZZZ(X)$ that corresponds to an ample class on $X$ and define $\h(X)$ as the set of all elements $v$ in $\ZZZ(X)$ such that $v\cdot v=1$ and $e_0\cdot v>0$. This yields a distance $d$ on $\h(X)$ by 
\[
d(u,v)\coloneqq \arccosh(u\cdot v).
\]
With this distance, the space $\h(X)$ is a complete metric space that is hyperbolic. 
The {\it boundary $\partial\h(X)$ of $\h(X)$} is the space of lines in the isotropic cone. 

Let $\pi\colon Y\to X$ be a birational morphism  between smooth projective surfaces. Then $\pi$ induces an isomorphism $\pi_*\colon \ZZ(Y)\to\ZZ(X)$ in the following way: Let $\NS(Y)=\NS(X)\oplus\Z e_{p_1}\oplus\dots\oplus \Z e_{p_n}$, where $p_1,\dots, p_n\in \B(X)$ are the points blown up by $\pi$, and $e_{p_i}$ is the irreducible component in the exceptional divisor that is contracted by $\pi$ to $p_i$. We now define the map $\pi_*$ by $\pi_*(e_p)=e_{\pi_\bullet(p)}$ for all $p\in\B(Y)$, $\pi_*(e_{p_i})=e_{p_i}$ and $\pi_*(D)=D$ for all $D\in \NS(X)\subset \NS(Y)$, with respect to the inclusion  given by the pull-back of $\pi$.  For a birational map $f\colon Y\dashrightarrow X$ we define an isomorphism $f_*\colon \ZZ(Y)\to\ZZ(X)$ by $f_*=(\pi_2)_*\circ (\pi_1)_*^{-1}$, where $\pi_1\colon Z\to Y, \pi_2\colon Z\to X$ is a minimal resolution of indeterminacies.  If $f\in\Bir(X)$, then $f_*$ induces an automorphism of $\ZZ(X)\otimes\R$, that extends to an automorphism of the completion $\ZZZ(X)$ and preserves the intersection form. This yields in particular an isometry on~$\h(X)$ and we obtain an action by isometries of $\Bir(X)$ on~$\h(X)$. 

Recall that there are three types of isometries of hyperbolic spaces: elliptic, parabolic and loxodromic isometries. Let $f$ be an isometry of $\h(X)$ and define $L(f)\coloneqq\inf \{d(f(p),p)\mid p\in\h(X)\}$. If $L(f)=0$ and the infimum is attained, i.e.\,$f$ has a fixed point in $\h(X)$, then $f$ is {\it elliptic}. If $L(f)=0$ but the infimum is not attained,  $f$  is {\it parabolic}. It can be shown that a parabolic isometry fixes exactly one point $p$ on the border $\partial \h(X)$. If $L(f)>0$ then $f$ is {\it loxodromic}. In this case the set
$\{p\in\mathbb{H}^{n-1}\mid d(h(p),p)=L(h)\}$
is a geodesic line in $\h(X)$, the so-called {\it axis} $\Ax(f)$ of $f$, and $L(f)$ is called the {\it translation length}. A loxodromic isometry has exactly two fixed points in $\partial\h(X)$, one of them attractive and the other one repulsive (see \cite{MR2811600}).
An element $f\in\Bir(X)$ is called  {\it elliptic}, {\it parabolic} or {\it loxodromic} if the corresponding isometry on $\h(X)$ is elliptic, parabolic or loxodromic respectively, and the {\it axis} $\Ax(f)$ of a loxodromic element $f\in\Bir(X)$ is the axis in $\h(X)$ of the  isometry corresponding to~$f$.

\subsection{Degrees}\label{degreesandtypes}
Let $X$ be a projective surface with a polarization $H$ and $f\in\Bir(X)$. The {\it dynamical degree} of $f$ is defined by
\[
\lambda(f)\coloneqq \lim_{n\to\infty}\deg_H(f^n)^{\frac{1}{n}}.
\]

The following result is well known (see for example \cite[Lemma 4.5]{cantatrev}):

\begin{proposition}
	The dynamical degree $\lambda(f)$ of a birational transformation $f\in\Bir(X)$ does not depend on the choice of the polarization $H$. Moreover, $f$ is loxodromic if and only if $\lambda(f)>1$. In this case, the translation length of the isometry of $\h(X)$ induced by $f$ is $\log(\lambda(f))$. 
\end{proposition}

In \cite{blanc2016dynamical}, Blanc and Cantat studied the spectrum of possible values  that can be obtained as dynamical degrees of birational transformations of a given projective surface.

\begin{theorem}[{\cite[Corollary 1.7]{blanc2016dynamical}}, \cite{diller2001dynamics}]\label{gapproperty}
	Let $X$ be a projective surface over an algebraically closed field $\kk$ and let $f\in\Bir(X)$ with $\lambda(f)>1$. Then $\lambda(f)$ is either a Pisot or a Salem number. Moreover, $\lambda(f)\geq\lambda_L$, where $\lambda_L>1$ is the Lehmer number.
\end{theorem}

Recall that the Lehmer number $\lambda_L\simeq 1.1762$ is the unique root $>1$ of the irreducible polynomial $x^{10}+x^9-x^7-x^6-x^5-x^4-x^3+x+1$. The fact that there is no birational transformation $f$ of a surface such that $1<\lambda(f)<\lambda_L$ is usually refered to as the {\it gap property}. One of the consequences of Theorem \ref{gapproperty} is the following:

\begin{theorem}[Cantat, Blanc, {\cite[Corollary 1.5]{blanc2016dynamical}}]\label{cantablancdeg}
	Two loxodromic elements $f,g\in\Cr_2(\kk)$ of degree $\leq d$ are conjugate if and only if they are conjugate by an element of degree $\leq (2d)^{57}$.
\end{theorem}

The dynamical behavior of a birational transformation $f$ of a surface $X$, in particular the growth of its degree under iteration, is closely linked to the type of the isometry of $\h(X)$ induced by $f$, as the following important theorem states (we refer to \cite{cantatrev} for details and references):

\begin{theorem}[Gizatullin; Cantat; Diller and Favre]\label{dim2}
	Let $X$ be a smooth projective surface over an algebraically closed field $\kk$ with a fixed polarization $H$ and $f\in\Bir(X)$. Then one of the following is true:
	\begin{enumerate}
		\item $f$ is elliptic, the sequence $\{\deg_H(f^n)\}$ is bounded  and there exists a $k\in\Z_+$ and a birational map $\varphi\colon X\dashrightarrow Y$ to a smooth projective surface $Y$ such that $\varphi f^k\varphi^{-1}$ is contained in $\aut^0(Y)$, the neutral component of the automorphism group $\aut(Y)$.
		\item[(2a)] $f$ is parabolic and $\deg_H(f^n)\sim cn$ for some positive constant $c$ and $f$ preserves a rational fibration, i.e.\,there exists a smooth projective surface $Y$, a birational map $\varphi\colon X\dashrightarrow Y$, a curve $B$ and a fibration $\pi\colon Y\to B$, such that a general fiber of $\pi$ is rational and such that $\varphi f\varphi^{-1}$ permutes the fibers of $\pi$.
		\item[(2b)] $f$ is parabolic and $\deg_H(f^n)\sim cn^2$ for some positive constant $c$ and $f$ preserves a fibration of genus 1 curves, i.e.\,there exists a smooth projective surface $Y$, a birational map $\varphi\colon X\dashrightarrow Y$, a curve $B$ and a fibration $\pi\colon Y\to B$, such that $\varphi f\varphi^{-1}$ permutes the fibers of $\pi$ and such that $\pi$ is an elliptic fibration, or a quasi-elliptic fibration (the latter only occurs if $\car(\kk)=2$ or $3$).
		\item[(3)] $f$ is loxodromic and $\deg_H(f^n)=c \lambda(f)^n+O(1)$ for some positive constant $c$, where $\lambda(f)$ is the dynamical degree of $f$. In this case, $f$ does not preserve any fibration.
	\end{enumerate}
\end{theorem}

Theorem \ref{dim2} has lead to various remarkable results on the group structure of $\Bir(X)$; we will state  some of them in the following sections. From the point of view of  geometric group theory, the plane Cremona group acting on $\h(X)$ has some analogies with other groups acting on hyperbolic spaces such as for example the mapping class group of a surface acting on the complex of curves or groups of outer automorphisms of a free group with $n$ generators acting on the outer space. 

\subsection{Groups preserving a fibration}
Let us recall the structure of subgroups of $\Cr_2(\kk)$ that preserve a given fibration.
	The {\it de Jonqui\`eres subgroup} $\J$ of $\Cr_2(\kk)$ is the subgroup of elements that preserve the pencil of lines through the point $[0:0:1]\in\p^2$. 
With respect to affine coordinates $[x:y:1]$, an element in $\J$ is of the form
\[
(x,y)\dashmapsto\left(\frac{ax+b}{cx+d},\frac{\alpha(x)y+\beta(x)}{\gamma(x)y+\delta(x)}\right),
\]
where $\left( \begin{array}{cc}
a & b \\ 
c & d
\end{array} \right) \in\PGL_2(\kk)$ and $\left( \begin{array}{cc}
\alpha(x) & \beta(x) \\ 
\gamma(x) & \delta(x)
\end{array} \right) \in\PGL_2(\kk(x))$.
This induces an isomorphism
\[
\J\simeq \PGL2(\kk)\ltimes\PGL_2(\kk(x)).
\]
By a Theorem of Noether and Enriques (\cite[III.4]{MR1406314}), every subgroup of $\Cr_2(\kk)$ that preserves a rational fibration is conjugate to a subgroup of $\J$.

Two smooth cubic curves $C$ and $D$ in $\p^2$ intersect in $9$ points $p_1,\dots, p_9$ and there is a pencil of cubic curves passing through these 9 points. By blowing up $p_1,\dots, p_9$, we obtain a rational surface $X$ with a fibration $\pi\colon X\to\p^1$ whose fibers are genus 1 curves. More generally, we can consider a pencil of curves of degree $3m$ for any $m\in\Z_+$ and blow up its base-points to obtain a surface $X$. Such a pencil of genus 1 curves is called a {\it Halphen pencil} and the surface $X$ a {\it Halphen surface of index $m$}. A surface $X$ is Halphen if and only if the linear system $|-mK_X|$ is one-dimensional, has no fixed component and is base-point free.  Up to conjugacy by birational maps, every pencil of genus 1 curves of $\p^2$ is a Halphen pencil and Halphen surfaces are the only examples of rational elliptic surfaces. We refer to \cite{MR2904576} and \cite[Chapter~10]{MR1392959} for proofs and more details. A birational transformation $f$ that preserves the genus 1-fibration of a Halphen surface $X$ preserves in particular the canonical divisor $K_X$ of $X$. This implies that $f$ is an automorphism. A subgroup $G$ of $\Cr_2(\kk)$ that preserves a pencil of genus 1 curves is therefore conjugate to a subgroup of the automorphism group of some Halphen surface.
The automorphism groups of Halphen surfaces are studied in \cite{MR563788} and in \cite{MR2904576}, see also \cite{MR3480704}. We need the following result, which can be found in \cite[Remark 2.11]{MR2904576}:

\begin{theorem}\label{halphenstructure}
	Let $X$ be a Halphen surface. Then there exists a homomorphism $\rho\colon \aut(X)\to\PGL_2(\C)$ with finite image such that $\ker(\rho)$ is an extension of an abelian group of rank $\leq 8$ by a cyclic group of order dividing $24$. 
\end{theorem}

We also recall the following result from \cite{MR2811600} (see also \cite[Lemma 2.5]{Urech:2018aa}):

\begin{lemma}\label{noloxfibration}
	Let $G\subset\Cr_2(\C)$ be a group that does not contain any loxodromic element but contains a parabolic element. Then $G$ is conjugate to a subgroup of the de Jonqui\`eres group $\J$ or to a subgroup of $\aut(Y)$, where $Y$ is a Halphen surface.
\end{lemma}

\subsection{Groups of elliptic elements}\label{zariskitop}
 The group $\Bir(X)$ can be equipped with the so called Zariski topology (see \cite{MR0284446}, \cite{serre2008groupe}, and \cite{MR3092478} for details), which is defined in the following way:
Let $A$ be an algebraic variety and 
\[
f\colon A\times X\dashrightarrow A\times X
\] 
a birational map of the form $(a,x)\dashmapsto (a,f(a,x))$ that induces an isomorphism between open subsets $U$ and $V$ of $A\times X$ such that the projections from $U$ and from $V$ to $A$ are both surjective. For each $a\in A$ we obtain therefore an element of $\Bir(X)$ defined by $x\mapsto p_2(f(a,x))$, where $p_2$ is the second projection. A map $A\to \Bir(X)$ of this form is called a {\it morphism}. The {\it Zariski topology} on $\Bir(X)$ is defined as the finest topology such that, for all algebraic varieties $A$, all the morphisms $f\colon A\to \Bir(X)$ are continuous (with respect to the Zariski topology on $A$).

\begin{theorem}[{\cite[Theorem 4.3]{MR3332894}}]\label{lowersemi1}
	The dynamical degree is a lower semi-continuous function. More precisely, 
	let $A\subset \Cr_2(\kk)$ be a family of birational transformations parametrized by an algebraic variety $A$. Then for all $\lambda\in\R$, the set $\{f\in A\mid\lambda(f)>\lambda\}$ is open in $A$. 
\end{theorem}

\begin{theorem}[{\cite[Theorem 1.6]{MR3332894}}]\label{lowersemi}
	Let  $d\geq 2$ be an integer. Denote by $\Cr_2(\kk)_d$ the space of Cremona transformations of degree $d$. Then for any $\lambda<d$, the set
	\[
	U_\lambda=\{f\in\Cr_2(k)_d\mid \lambda_1(f)>\lambda\}
	\]
	is open and Zariski-dense in the algebraic variety $\Cr_2(\kk)_d$.
\end{theorem}

An {\it algebraic subgroup} of $\Bir(X)$ is the image of an algebraic group $G$ by a morphism $G\to \Bir(X)$ that is also an injective group homomorphism. Algebraic groups are closed in the Zariski topology and of bounded degree in the case of $\Bir(X)=\Cr_n(\kk)$. Conversely, closed subgroups of bounded degree in $\Cr_n(\kk)$ are always algebraic subgroups with a unique algebraic group structure that is compatible with the Zariski topology (see \cite{MR3092478}). In \cite{MR3092478}, it is shown moreover, that all algebraic subgroups of $\Cr_n(\kk)$ are linear.

Every algebraic subgroup of $\Cr_2(\C)$ is contained in a maximal algebraic subgroup. The maximal algebraic subgroups of $\Cr_2(\C)$ have been classified by Enriques and Blanc. Each of them can be realized as an automorphism group of a complex projective variety (\cite{enriques1893sui}, \cite{MR2504924}):

\begin{theorem}[\cite{MR2504924}]\label{maxsubgroups}
	Every algebraic subgroup of $\Cr_2(\C)$ is contained in a maximal algebraic subgroup. The maximal algebraic subgroups of $\Cr_2(\C)$ are conjugate to one of the following groups:
	\begin{enumerate}
		\item $\aut(\p^2)\simeq \PGL_3(\C)$
		\item $\aut(\p^1\times\p^1)\simeq (\PGL_2(\C))^2\rtimes \Z/2\Z$
		\item $\aut(S_6)\simeq(\C^*)^2\rtimes (\s_3\times \Z/2\Z)$, where $S_6$ is the del Pezzo surface of degree~6.
		\item $\aut(\F_n)\simeq\C[x,y]_n\rtimes\GL_2(\C)/\mu_n$, where $n\geq 2$ and $\F_n$ is the $n$-th Hirzebruch surface and $\mu_n\subset\GL_2(\C)$ is the subgroup of $n$-torsion elements in the center of $\GL_2(\C)$. 
		\item $\aut(S,\pi)$, where $\pi\colon S\to\p^1$ is an exceptional conic bundle.
		\item[(6)-(10)] $\aut(S)$, where $S$ is a del Pezzo surface of degree $5$, $4$, $3$, $2$ or $1$. In this case, $\aut(S)$ is finite.
		\item[(11)] $\aut(S,\pi)$, where $(S,\pi)$ is a $(\Z/2\Z)^2$-conic bundle and $S$ is not a del Pezzo surface. There exists an exact sequence
		\[
		1\to V\to\aut(S,\pi)\to H_V\to 1,
		\]
		where $V\simeq (\Z/2\Z)^2$ and $H_V\subset\PGL_2(\C)$ is a finite subgroup.
	\end{enumerate}
\end{theorem}

A subgroup of $\Cr_2(\C)$ consisting only of elliptic elements is called a {\it group of elliptic elements}. In \cite{Urech:2018aa} groups of elliptic elements of $\Cr_2(\C)$ have been classified. In particular, the following result is shown:

\begin{theorem}[{\cite[Theorem 1.1 and 1.2]{Urech:2018aa}}]\label{ellipticelements}
	Let $G\subset\Cr_2(\C)$ be a subgroup of elliptic elements. Then one of the following is true:
	\begin{enumerate}
		\item $G$ is conjugate to a subgroup of an algebraic group;
		\item $G$ preserves a rational fibration;
		\item $G$ is a torsion group and $G$ is isomorphic to a subgroup of an algebraic group.
	\end{enumerate}
\end{theorem}

We also recall the following result, which in the original version was stated for the case of complex numbers. However, the proof only relies on hyperbolic geometry of $\h(\p^2)$ and does not depend on the characteristic of the base field:

\begin{theorem}[{\cite[Proposition 6.14]{MR2811600}}]\label{cantatelliptic}
	Let $\kk$ be an algebraically closed field and let $\Gamma\subset\Cr_2(\kk)$ be a finitely generated subgroup of elliptic elements. Then $\Gamma$ is either contained in an algebraic subgroup, or $\Gamma$ preserves a rational fibration and is therefore conjugate to a subgroup of $\J\simeq\PGL_2(\kk)\ltimes\PGL_2(\kk(t))$.
\end{theorem}

\subsection{Monomial transformations}
The subgroup of diagonal automorphisms $D_2\subset\PGL_{3}(\kk)$ is a torus of rank $2$. It is maximal in the following sense: all algebraic tori in $\Cr_2(\kk)$ are of rank $\leq 2$ and are conjugate in $\Cr_2(\kk)$ to a subtorus of $D_2$ (\cite{bialynicki1967remarks}, \cite{MR0284446}).
A matrix $A=(a_{ij})\in\GL_2(\Z)$ determines a rational map $f_A$ of $\p^2$, which we define by
\[
f_A=(x^{a_{11}}y^{a_{12}},x^{a_{21}}y^{a_{22}}).
\] 
We thus obtain an injective homomorphism $\GL_2(\Z)\to\Cr_{2}(\kk)$. By abuse of notation, we will identify its image with $\GL_2(\Z)$. The normalizer of $D_2$ in $\Cr_2(\kk)$ is the semidirect product 
\[
\norm_{\Cr_2(\kk)}(D_2)=\GL_2(\Z)\ltimes D_2.
\]
Elements in $\GL_2(\Z)\ltimes D_2$ are called {\it monomial transformations}. We say that $f\in\Cr_2(\kk)$ is of {\it monomial type}, if $f$ is conjugate to an element in $\GL_2(\Z)\ltimes D_2$. We call a matrix $A\in\GL_2(\Z)$ {\it loxodromic}, if the corresponding birational monomial map in $\Cr_2(\kk)$ is loxodromic. 

\begin{lemma}\label{puremonomial}
	Let $m\in\GL_2(\Z)\subset\Cr_2(\kk)$ be a loxodromic monomial transformation and $d\in D_2$ a diagonal automorphism. There exists a diagonal automorphism $d'\in D_2$ such that $d'^{-1}dmd'=m$.
\end{lemma}

\begin{proof}
	Assume that $m=(x^ay^b,x^cy^d)$, where $A\coloneqq\left(\begin{matrix}
		a & b \\ 
		c & d
	\end{matrix} \right)\in\GL_2(\Z)$. Then $m$ acts by conjugation on $D_2$ by sending $(c_1x,c_2y)$ to $(c_1'x, c_2'y)$, where $c_1'=c_1^ac_2^b$ and $c_2'=c_1^cc_2^d$. We therefore have $(c_1x,c_2y)^{-1}m(c_1x,c_2y)=(d_1x, d_2y)m$, where $d_1=c_1^{a-1}c_2^b$ and $d_2=c_1^cc_2^{d-1}$. To show the claim of the lemma it is therefore enough to show that the homomorphism $\varphi_{A-\id}$ of $D_2$ given by  $(c_1x,c_2y)\mapsto (d_1x, d_2y)$ is surjective. Since  $m$ is loxodromic, the matrix $A$ has no eigenvalue of modulus $1$ and hence the determinant of $A-\id$ is not $0$. This is equivalent to the kernel of $\varphi_{A-\id}$ being finite, which implies surjectivity.
\end{proof}

\begin{lemma}\label{densed2}
	Let $m\in\GL_2(\Z)\subset\Cr_2(\kk)$ be a loxodromic monomial map and $\Delta_2\subset D_2$ an infinite subgroup that is normalized by $m$. Then $\Delta_2$ is dense in $D_2$ with respect to the Zariski topology. 
\end{lemma}

\begin{proof}
	Let $\overline{\Delta}_2^0$ be the neutral component of the Zariski-closure of $\Delta_2$. If $\overline{\Delta}_2^0$ has dimension 2, we are done. Otherwise, $\overline{\Delta}_2^0$ is of dimension 1, since $\Delta_2$ is not finite. It follows that a general orbit of $\overline{\Delta}_2^0$ has dimension 1. Since $m$ normalizes $\overline{\Delta}_2^0$ it permutes the orbits. This implies that $m$ preserves a fibration. But this is not possible since $m$ is loxodromic, by assumption.
\end{proof}

\begin{lemma}\label{uniquembeddtorus}
	Let $f\in\Cr_2(\kk)$ be a birational transformation such that $fD_2f^{-1}\subset \GL_2(\Z)\ltimes D_2$. Then $f\in \GL_2(\Z)\ltimes D_2$.
\end{lemma}

\begin{proof}
	Since $fD_2f^{-1}$ is an algebraic subgroup, it is of bounded degree. Since $\GL_2(\Z)$ contains only finitely many elements of a given degree, $fD_2f^{-1}$ is therefore contained in a group of the form $H\ltimes D_2$, where $H$ is finite. Since $fD_2f^{-1}$ is connected, it is contained in $D_2$, which implies that $f$ normalizes $D_2$.
\end{proof}

Let $M\in\GL_2(\Z)$ and $f_M$ be the corresponding birational transformation. The dynamical degree $\lambda(f_M)$ of $f_M$ is exactly the spectral radius of the matrix $M$. This shows in particular that the dynamical degree of a monomial matrix is always a quadratic algebraic integer.
If $M\in\GL_2(\Z)$ has spectral radius strictly larger than 1, the birational map $f_M$ is loxodromic. This yields examples of loxodromic elements that normalize an infinite subgroup consisting only of elliptic elements. The following theorem shows that, up to conjugacy, these are the only examples with this property if we work over the field of complex numbers $\C$. A first version has been proven by Cantat in {\cite[Theorem 7.1]{MR2881312}}, the more general version, which we state below, can be found in \cite{Urech:2018aa}:

\begin{theorem}\label{ellitpicnorm}
	Let $N\subset\Cr_2(\C)$ be a subgroup containing at least one loxodromic element. Assume that there exists a short exact sequence 
	\[
	1\to A\to N\to B\to 1,
	\]
	where $A$ is an infinite group of elliptic elements. Then $N$ is conjugate to a subgroup of $\GL_2(\Z)\ltimes D_2$.
\end{theorem}

\subsection{Small cancellation}
Small cancellation has been one of the fundamental tools used by Cantat and Lamy to show that $\Cr_2(\C)$ is not simple. Dahmani, Guirardel and Osin applied similar arguments in the context of mapping class groups (\cite{MR3589159}). We refer to \cite{MR3522169} for an overview of the subject.

Let $\epsilon,B>0$. We say that two geodesic lines $L$ and $L'$ in $\h$ are {\it $(\epsilon, B)$-close}, if the diameter of the set $S=\{x\in L\mid d(x, L')\leq \epsilon\}$ is at least $B$.

\begin{definition}
	Let $G$ be a subgroup of $\Cr_2(\kk)$ and $g\in G$ a loxodromic element. We call $g$  {\it rigid in $G$} if there exists an $\epsilon>0$ and a $B>0$ such that for every element $h\in G$ one has: $h(\Ax(g))$ is $(\epsilon, B)$-close to $\Ax(g)$ if and only if $h(\Ax(g))=\Ax(g)$.
	
	We call $g$ {\it tight in $G$} if it is rigid in $G$ and if $h(\Ax(g))=\Ax(g)$ implies $hgh^{-1}=g$ or $hgh^{-1}=g^{-1}$  for all $h\in G$ .
\end{definition}

\begin{example}\label{monomialexample}
	Let $m\in \GL_2(\Z)\subset\Cr_2(\kk)$ be a loxodromic element. Then the group $D_2$ fixes the axis of $m$ pointwise and no power of $m$ is tight in $\Cr_2(\kk)$ (\cite[Example 7.1]{cantatrev}). More generally, if $G\subset\Cr_2(\kk)$ is a subgroup containing $m$ and an infinite subgroup $\Delta_2\subset D_2$, then no power of $m$ is tight in $G$.
\end{example}

\begin{lemma}\label{loxogllemma}
	Let $g\in\GL_2(\Z)$ be a loxodromic element and $f\in\Cr_2(\kk)$ an element such that $fgf^{-1}$ is contained in $\GL_2(\Z)\ltimes D_2$. Then $f\in \GL_2(\Z)\ltimes D_2$.
\end{lemma}

\begin{proof}
	Assume that $fgf^{-1}\in \GL_2(\Z)\ltimes D_2$. Then the axis of the loxodromic element $fgf^{-1}$ is fixed pointwise by both, $fD_2f^{-1}$ and $D_2$ (see Example \ref{monomialexample}). Hence, the group $A$ generated by $fgf^{-1}$ and $D_2$ is bounded. By Theorem \ref{ellitpicnorm}, $A$ is conjugate to a subgroup of $D_2$. This implies that $fD_2f^{-1}\subset D_2$ and therefore, that $f\in\GL_2(\Z)\ltimes D_2$.
\end{proof}

 In \cite{cantat2013normal} the authors show  that a power of a generic element in $\Cr_2(\C)$ is tight in $\Cr_2(\C)$ and prove the following theorem to argue that $\Cr_2(\C)$ is not simple. In \cite{MR3533276} Lonjou showed that a power of a H\'enon transformation is tight in $\Cr_2(\kk)$, where $\kk$ is an arbitrary field in order to show that $\Cr_2(\kk)$ is not simple. 

\begin{theorem}[{\cite[Theorem 2.10]{cantat2013normal}}]\label{smallcanc}
	Let $G\subset \Cr_2(\C)$ be a subgroup and let $g\in G$ be an element that is tight in $G$. Then every element $h$ in $\left<\left<g \right>\right>$, where $\left<\left<g \right>\right>$ denotes the normal subgroup of $G$ generated by $g$, satisfies the following alternative: Either $h=\id$  or $h$ is  loxodromic and $L(h)\geq L(g)$. In particular, for $n\geq 2$, the element $g$ is not contained in $\left<\left<g^n \right>\right>$ and $G$ is therefore not simple.
\end{theorem}

A couple of years later, Shepherd-Barron has classified tight elements in $\Cr_2(\C)$ using Theorem \ref{ellitpicnorm}:

\begin{theorem}[\cite{Shepherd-Barron:2013qq}]\label{shep}
	 In $\Cr_2(\C)$ every loxodromic element is rigid. If $g$ is conjugate to a monomial map, then no power of $g$ is tight. In all the other cases, there exists an integer $n$ such that $g^n$ is tight.
\end{theorem}

Note that if $G\subset\Cr_2(\C)$ is a subgroup and $g\in \Cr_2(\C)$ is a rigid element, then $g$ is rigid in $G$ as well. The same is true for tight elements. However, there exist subgroups $G\subset\Cr_2(\C)$ and loxodromic elements $g\in G$ such that $g$ is tight in $G$ but not in $\Cr_2(\C)$. From the proof of Theorem \ref{shep} (see \cite[p.18]{Shepherd-Barron:2013qq}) and Lemma~\ref{densed2} the following Theorem follows:

\begin{theorem}\label{sheprel}
	Let $G\subset\Cr_2(\C)$ be a subgroup and $g\in G$ a loxodromic element. The following two conditions are equivalent:
	\begin{enumerate}
		\item no power of $g$ is tight in $G$;
		\item there is a subgroup $\Delta_2\subset G$ that is normalized by $g$ and a birational transformation $f\in\Cr_2(\C)$ such that $f\Delta_2f^{-1}\subset D_2$ is a dense subgroup and $fgf^{-1}\in\GL_2(\Z)\ltimes D_2$.
	\end{enumerate}
\end{theorem}

From Theorem \ref{smallcanc} and Theorem \ref{sheprel} one deduces directly the following lemma:

\begin{lemma}\label{sheplemma}
	Let $G\subset\Cr_2(\C)$ be a simple subgroup. Then for every loxodromic element $g\in G$ there exists an infinite subgroup $\Delta_2^g\subset G$ and an element $f\in\Cr_2(\C)$ that conjugates $\Delta_2^g$ to a dense subgroup of $D_2$ and $g$ to an element of $\GL_2(\Z)\ltimes D_2$. 
\end{lemma}

In positive characteristic Theorem \ref{ellitpicnorm} does no longer hold, as the following example shows:

\begin{example}
	Let $\kk$ be a field of characteristic $p$. The loxodromic element $(y, x+y^p)$ normalizes the additive group of elliptic elements $\kk^2$. 
\end{example} 

However, it turns out that these kind of examples are the only exceptions (see \cite{Shepherd-Barron:2013qq} and \cite{cantatrev}). We only need the following result, which follows from the proof of Theorem 7.2 in  \cite[p.18]{Shepherd-Barron:2013qq}:

\begin{theorem}
	Let $\kk$ be an algebraically closed field and $G\subset\Cr_2(\kk)$ a subgroup. Let $g\in G$ be a loxodromic element such that no power of $g$ is tight then $g$ normalizes an infinite group of elliptic elements that is either conjugate to a subgroup of $D_2$ or to a subgroup of $\kk^2$.
\end{theorem}

\subsection{Non-rational surfaces}\label{nonrational}

In this section we treat the case of non-rational compact complex K\"ahler surfaces of Theorem \ref{mainthm}. 

\begin{lemma}\label{negativekodaira}
	Let $S$ be a non-rational compact complex K\"ahler surface of Kodaira dimension $-\infty$ and $G\subset \Bir(S)$ a simple subgroup. Then $G$ is either finite or isomorphic to a subgroup of $\PGL_2(\C)$.
\end{lemma}

\begin{proof}
	There exists a non-rational curve $C$ such that $S$ is birationally equivalent to $\p^1\times C$, hence $\Bir(S)=\PGL_2(\C(C))\rtimes \aut(C)$. It follows therefore that $G\subset\PGL_2(\C(C))$ or $G\subset \aut(C)$. In the first case we are done, since the function field $\C(C)$ can be embedded into $\C$ and hence $\PGL_2(\C(C))$ is a subgroup of $\PGL_2(\C)$. In the second case the lemma follows since $\aut(C)$ is either finite or contains a normal abelian subgroup of finite index. 
\end{proof}

\begin{lemma}\label{nonnegativekodaira}
	Let $S$ be a compact complex K\"ahler surface of non-negative Kodaira dimension and let $G\subset \Bir(S)$ be a simple subgroup. Then $G$ is finite.
\end{lemma}

\begin{proof}
The class of compact complex K\"ahler surfaces that are birationally equivalent to $S$ contains a unique smooth minimal model $S'$. It follows that $\Bir(S)\simeq\Bir(S')=\aut(S')$. The group $\aut(S')$ acts by linear transformations on the cohomology. This gives a linear representation $\varphi\colon \aut(S')\to\GL(H^*(S'; \Z))$, where $H^*(S'; \Z)$ is the direct sum of the cohomology groups of $S'$. The kernel of $\varphi$ is an algebraic group (see \cite{MR521918}) and hence and extension of a complex torus by a complex linear algebraic group. Let $G\subset\aut(S')$ be a simple group, then either $G$ is contained in $\GL_n(\Z)$ for some $n$, and therefore finite. Or $G$ is isomorphic to a subgroup of an algebraic group. Since $S'$ is of non-negative Kodaira dimension, there are no linear algebraic groups of positive dimension operating on $S'$, since otherwise $S'$ would be uniruled.  Hence $H$ is abelian up to finite index and therefore $G$ is finite.
\end{proof}

\section{Subgroups containing loxodromic elements}\label{loxosection}
In all of Section \ref{loxosection} we always work over the field $\C$ of complex numbers. The main result of this section is the following:

\begin{theorem}\label{loxomain}
	A simple subgroup $G\subset\Cr_2(\C)$ does not contain any loxodromic element.
\end{theorem}

The starting point to prove Theorem \ref{loxomain} is Lemma \ref{sheplemma}. It states that all loxodromic elements in a simple group $G$ are of monomial type and that, up to conjugation, $G$ contains a dense subgroup $\Delta_2\subset D_2$ for each of its loxodromic elements. Our strategy is to show that if $G$ contains a loxodromic element, these conditions imply that $G$ is conjugate to a subgroup of $\GL_2(\Z)\ltimes D_2$ and from this we will deduce a contradiction to the simplicity of $G$.

In Section~\ref{degreebounds} we first prove a result about the degrees of elements that conjugate loxodromic elements to monomial elements. In Section \ref{basepoints} we take a closer look at the dynamical behavior of exceptional curves and base-points. This will allow us to prove Theorem \ref{loxomain} in Section \ref{proofofloxomain}.

\subsection{Degree bounds}\label{degreebounds}
We start with some facts about loxodromic monomial elements.

\begin{lemma}\label{conjugpos}
	Let $A\in\SL_2(\Z)$ be a loxodromic element. Then either $A$ or $-A$ is conjugate in $\GL_2(\Z)$ to a matrix $B$ such that all entries of $B$ are non-negative.
\end{lemma}

\begin{proof}
	In \cite[Theorem 7.3]{MR3099298} it is shown that  for a loxodromic element $A\in\SL_2(\Z)$ either $A$ or $-A$ is conjugate to a matrix of the form $\left(\begin{array}{cc}
		a & c \\ 
		b & d
	\end{array} \right)$, where $d>b>a\geq 0$. If $a\neq 0$ then $ad-bc=1$ implies $c\geq 0$ and we are done. If $a=0$, then we calculate 
	\[\left(\begin{array}{cc}
	1 & 1 \\ 
	0 & 1
	\end{array} \right)\left(\begin{array}{cc}
		0 & c \\ 
		b & d
	\end{array} \right)\left(\begin{array}{cc}
	1 & -1 \\ 
	0 & 1
	\end{array} \right)=\left(\begin{array}{cc}
	b & c+d-b \\ 
	b & d-b
	\end{array} \right).\]
	Since $d-b$ and $b$ are positive, it follows that $c+d-b\geq 0$. If $c+d-b= 0$ the matrix is not loxodromic anymore, hence $c+d-b>0$ and we are done.
\end{proof}

With the help of Lemma \ref{conjugpos} the following well-known Lemma can be proved:

\begin{lemma}\label{conjuggl}
	For an integer $n\in\Z$ there exist only finitely many conjugacy classes of loxodromic elements in $\GL_2(\Z)$ with trace $n$. 
\end{lemma}

\begin{lemma}\label{conjugw}
	Let $\lambda>0$ and $g\in \GL_2(\Z)\ltimes D_2\subset\Cr_2(\C)$. If $\lambda_1(g)\leq\lambda$ then $g$ is conjugate in $\GL_2(\Z)$ to an element of degree $\leq C(\lambda)$, where the constant $C(\lambda)$ only depends on $\lambda$.
\end{lemma}

\begin{proof}
	By Lemma \ref{puremonomial}, we may assume that $gh\in\GL_2(Z)$. The dynamical degree $\lambda_1(g)$ is the spectral radius of $g$, i.e.\,the absolute value of the eigenvalue of the matrix $g$ that is strictly larger than 1. The condition $\lambda_1(g)\leq\lambda$ implies that $|\tr(g)|=|\lambda_1(g)+\lambda_1(g)^{-1}|\leq \lambda+1$. So $\tr(g)$ is contained in the finite set of integers between $-(\lambda+1)$ and $(\lambda+1)$. By Lemma \ref{conjuggl}, there exist only finitely many conjugacy classes in $\GL_2(\Z)$ to which $g$ can belong. Denote by $f_1,\dots, f_n$ representants of these classes. We set $C(\lambda)=\max\{\deg(f_1),\dots,\deg(f_n)\}$.  
\end{proof}

Finally, we are able to prove the main result of this section:

\begin{lemma}\label{degreebound}
	Let $g\in\Cr_2(k)$ be a loxodromic element of monomial type. Then there exists an $m\in \GL_2(\Z)\ltimes D_2$ and a  constant $K$ depending only on $d:=\deg(g)$, such that $g$ is conjugate to $m$ by an element of degree $\leq K$.
\end{lemma}

\begin{proof}
	We observe that $\lambda_1(g)\leq d$. By Lemma \ref{conjugw}, there exists a constant $C(d)$ such that $g$ is conjugate to an $m\in W_2\ltimes D_2$ of degree $\leq C(d)$. By Theorem~\ref{cantablancdeg}, $g$ can be conjugated to $m$ by an element of degree $\leq K$, where $K=(2r)^{57}$ for $r=\max\{d, C(d)\}$.
\end{proof}

\subsection{Base-points and toric boundaries}\label{basepoints}

Let $S$ be a smooth projective surface with a given regular $D_2$-action that has an open orbit $U\subset S$. The fixed points of this action are called {\it toric points}, the algebraic set $\partial S\coloneqq S\setminus U$ is called the {\it toric boundary}. In what follows, we consider $\p^2$ equipped with the standard action of $D_2$, or blow-ups of toric points $\pi\colon S\to\p^2$  with the pull-back of the standard action of $D_2$ on $\p^2$. In this case, $\partial S$ is always a curve whose irreducible components are curves isomorphic to $\p^1$ with self-intersection $\leq 1$.

A {\it toric point} in the bubble space $\mathcal{B}(\p^2)$ is a point of the form $(p, S, \pi)$, where $\pi\colon S\to\p^2$ is the blow-up of toric points and $p\in S$ is a toric point. If a toric point $q_1\in\mathcal{B}(\p^2)$ lies above a point $q_2\in\mathcal{B}(\p^2)$, then $q_2$ is toric as well. 

Let $C$ be a curve on a surface $S$ and $f\in\Bir(S)$. In what follows, we denote by $f(C)$ the strict transform of $C$ under $f$, i.e.\,the closure of $f(C\setminus \{\Ind(f)\})$, and by $f^{-1}(C)$ the strict transform of $C$ under $f^{-1}$. Note that with this notation, $f^{-1}(C)$ does not contain all the points that are mapped by $f$ to $C$.

Let $S$ be a projective surface, $f\in\Bir(S)$, and assume that $f$ contracts a curve $C\subset S$. If $f(C)=p\in S$ we say that $f$ {\it contracts $C$ to $p$}. We extend this notion to infinitely near points. Consider a point in the bubble space $\mathcal{B}(S)$ with a representative $(p, T, \pi)$. Let $\tilde{f}\in\Bir(T)$ be given by $\tilde{f}\coloneqq\pi^{-1} f\pi$ and denote by $\tilde{C}$ the strict transform of $C$ under $\pi$. We say that $f$ {\it contracts $C$} to $p$ if $\tilde{f}(\tilde{C})=p$. If $p$ lies above a point $q$ in $\mathcal{B}(S)$ and $f$ contracts a curve $C\subset S$ to $p$, then $f$ also contracts $C$ to $q$. Note as well, that if a birational transformation $f\in\Cr_2(\C)$ contracts a curve $C$ to a non-toric point in $\mathcal{B}(\p^2)$, then there exists a blow-up of toric points $\pi\colon S\to \p^2$ such that $\pi^{-1}f\pi(C)$ is a proper non-toric point of $S$. 

By abuse of notation, in this section  we will sometimes denote the lift of a birational transformation $f\in \Cr_2(\C)$ under a blow-up of toric points $\pi\colon S\to\p^2$ again by $f$. This will simplify the notation, as the choice of $\pi$ will always be clear from the context. Similarly, we will identify a curve $C$ on $\p^2$ with its strict transform on $S$, if there is no ambiguity.

\begin{definition}
	Let $S$ be a projective surface and $f\in\Bir(S)$. We denote by $E(f)$ the number of irreducible components of the exceptional divisor of $f$.
\end{definition}

\begin{remark}\label{remarkei}
	For $f\in\Cr_2(\C)$, the numbers $E(f)$ can be bounded by a constant depending only on the degree of $f$. If $f$ and $g$ are two Cremona transformation, then $E(fg)\leq E(f)+E(g)$.
\end{remark}

\begin{lemma}\label{boundedei}
	Let $S$ be a rational projective surface, $f\in\Bir(S)$ of monomial type and $\pi\colon S\dashrightarrow\p^2$ a birational transformation. Then $E(f^n)$ is uniformly bounded for all $n$ by a constant $K$ only depending on $\pi$ and the degree of $\pi f \pi^{-1}$.
\end{lemma}

\begin{proof}
	The birational transformation $\pi\colon S\dashrightarrow \p^2$ only contracts finitely many irreducible curves. So $E(f^n)$ is uniformly bounded for all $n$ if and only if $E(\pi f^n\pi^{-1})$ is uniformly bounded. It is therefore enough to consider the case $f\in\Bir(\p^2)$.
	
	By Lemma \ref{degreebound}, there exists a $g\in\Cr_2(\C)$ of degree $\leq C$, where $C$ only depends on $\deg(f)$, such that $gfg^{-1}=m\in\GL_2(\Z)\ltimes D_2$. We have $E(m^n)\leq 3$  for all $n$.  By Remark \ref{remarkei}, $E(g)$ is bounded by a constant $K'$ depending only on $\deg(g)$ and hence only on $\deg(f)$. Therefore, $E(f^n)=E(gm^ng^{-1})\leq 2K'+3$ and we thus set $K\coloneqq 2K'+3$.
\end{proof}

\begin{lemma}\label{loxoconj}
	Let $f\in\Cr_2(\C)$ be a loxodromic element that is not contained in $\GL_2(\Z)\ltimes D_2$. There exists an $n\in\Z_+$ and a dense open set $V\subset D_2$ such that $f^nd^{-1}f^{-n}d$ is loxodromic for each $d\in V$.
\end{lemma}

\begin{proof}
	Let $\alpha^+\in\partial \h$ be the attracting fixed point of the isometry of $\h$ induced by $f$, and let $\alpha^-\in \partial\h$ be its repulsive fixed point. The axis $\Ax(f)$ is the geodesic line between $\alpha^+$ and $\alpha^-$. We claim that there exists a dense open subset $U\subset D_2$ of elements that fix neither $\alpha^+$ nor $\alpha^-$. Denote by $G\subset \Cr_2(\C)$ the subgroup of all elements that fix $\alpha^+$. Let $L\subset\ZZZ(\p^2)$ be the one-dimensional subspace  that corresponds to $\alpha^+$. Since $G$ fixes $\alpha^+$, its linear action on $\ZZZ(\p^2)$ restricts to an action on $L$ by automorphisms preserving the orientation. This yields a group homomorphism $\rho\colon G\to\R_+^*$.  Loxodromic elements don't fix any vector in $\ZZZ(\p^2)$. Let us note as well that the group $G$ does not contain any parabolic element since $\alpha^+$ is fixed by a loxodromic element and does therefore not correspond to the class of a fibration. It follows that the kernel of $\rho$ is a subgroup of elliptic elements, which is normalised by $f$. If $\ker(\rho)$ is infinite, there exists, by Theorem \ref{ellitpicnorm}, an element $h\in\Cr_2(\C)$, such that $hGh^{-1}\subset\GL_2(\Z)\ltimes D_2$. As $f$ is not in $\GL_2(\Z)\ltimes D_2$, the transformation $h$ is not in $\GL_2(\Z)\ltimes D_2$ and therefore, by Lemma~\ref{uniquembeddtorus}, $h^{-1}D_2h\cap D_2$ is a proper closed subset of $D_2$. In particular, there exists a dense open set $U_1\subset D_2$ that is not contained in $G$. If $\ker(\rho)$ is finite, the existence of such a dense open $U_1\subset D_2$ follows trivially. With the same argument, we obtain a dense open set $U_2\subset D_2$ that does not fix $\alpha^-$. Define $U\coloneqq U_1\cap U_2$. This proves the claim.
	
	Let $U^2=\{d^2\mid\in U\} $ and let $d\in U\cap U^2$ be arbitrary. Then $d$ does neither fix $\alpha^+$ nor $\alpha^-$ and $d(\alpha^+)\neq\alpha^-$. Denote by $\beta^+\in\partial \h$ the attracting fixed point of the loxodromic isometry $d^{-1}f^{-1}d$  and by $\beta^-\in\partial\h$ its repulsive fixed point. By the above observation, $\alpha^+, \alpha^-, \beta^+$ and $\beta^-$ are pairwise disjoint. Let $S_1^+$ be a small neighborhood of $\alpha^+$ in $\partial\h$ and $S_1^-$ a small neighborhood of $\alpha^-$. Similarly, let $S_2^+$ be a small neighborhood of $\beta^+$ and $S_2^-$ a small neighborhood of $\beta^-$. We may assume that $S_1^+, S_1^-, S_2^+$ and $S_2^-$ are pairwise disjoint. Since $\beta^+$ is attractive, there exists an $n_1\in\Z_+$ such that $d^{-1}f^{-n_1}d(S_1^+)\subset S_2^+$. Similarly, let $n_2\in\Z_+$ be such that $f^{n_2}(S_2^+)\subset S_1^+$ is a proper subset. For $n\coloneqq\max\{n_1, n_2\}$, we obtain that $f^nd^{-1}f^{-n}d(S_1^+)$ is a proper subset of $S_1^+$. Analogously, if we choose $n$ large enough, $(f^nd^{-1}f^{-n}d)^{-1}(S_2^-)$ is a proper subset of $S_2^+$. Thus, $f^nd^{-1}f^{-n}d$ has an attractive fixed point in $S_1^+$ and a repulsive fixed point in $S_2^-$. In particular, $f^nd^{-1}f^{-n}d$ is loxodromic.
	
	Consider the family of birational transformations $\{f^nd^{-1}f^{-n}d\mid d\in D_2\}$. It contains one element of dynamical degree $\lambda>1$. By Theorem \ref{lowersemi1}, the dynamical degree is a lower semi-continuous function. Hence, there exists a dense open subset $V\subset D_2$ such that the dynamical degree of $f^nd^{-1}f^{-n}d$ is $>1$ for all $d\in V$, which is equivalent to $f^nd^{-1}f^{-n}d$ being loxodromic.
\end{proof}

\begin{lemma}\label{loxomonpres}
	Let $m$ be a loxodromic monomial transformation and $(x,y)$ affine coordinates. Let $L_x$ be the line given by $x=0$ and $L_y$ be the line given by $y=0$. Then $m(L_x)\neq L_x$ and $m(L_y)\neq L_y$.
\end{lemma}

\begin{proof}
	By Lemma \ref{puremonomial}, we may assume that $m\in\GL_2(\Z)$, since $d(L_x)=L_x$ and $d(L_y)=L_y$ for all $d\in D_2$. 
	
	It is now enough to observe that $m(L_x)=L_x$ implies that $m$ is of the form $(xy^k,y^{\pm 1})$ and $m(L_y)=L_y$ implies that $m$ is of the form $(x^{\pm 1},x^ky)$ for some $k\in\Z$. No transformation of the form $(xy^k,y^{\pm 1})$ or of the form $(x^{\pm 1},x^ky)$ is loxodromic.
\end{proof}

\begin{lemma}\label{monomiallemmanofix}
	Let $m$ be a loxodromic monomial transformation and $\pi\colon S\to \p^2$ a blow-up of toric points. Let $L\subset\partial S$ be an irreducible boundary component, then $\pi^{-1}m\pi(L)\neq L$.
\end{lemma}

\begin{proof}
	Assume that there exists a blow-up of toric points $\pi\colon S\to\p^2$ and a line $L\subset \partial S$ such that $\pi^{-1}m\pi(L)=L$. By Lemma \ref{loxomonpres}, the line $L$ is not the strict transform of a line in $\p^2$.  After possibly contracting components of $\partial S$ different from $L$, we can write $\pi=\pi_k\circ\pi_{k-1}\circ\cdots\circ\pi_1$, where each $\pi_i\colon S_{i-1}\to S_{i}$ is the blow-up of a single toric point $p_i$ such that $\pi_1(L)=p_1$ and $\pi_l\circ\cdots\pi_1(L)=\pi_l(p_{l-1})=p_l$ for all $1\leq l\leq k$. We have $S_0=S$ and $S_k=\p^2$. Since $\pi^{-1}m\pi(L)=L$, we have that $p_l$ is a fixed point of the pull-back of $m$ on $S_l$ for all $1\leq l\leq k$. 
	
	Let $(x_0,y_0)$ be local affine coordinates of $S$ such that $L$ is defined by $x_0=0$ and let $(x_1, y_1)$ be local affine coordinates of $S_1$ such that $p_1=(0,0)$ and the exceptional divisor of $\pi_2$ is given by $x_2=0$. We proceed inductively and define local affine coordinates $(x_l, y_l)$ of $S_l$ in such a way that $p_l=(0,0)$ and the exceptional divisor of $\pi_{l+1}$ is given by $x_l=0$ for all $1\leq l\leq k$. With respect to the local affine coordinates $(x_l, y_l)$, the blow-up $\pi_{l}\colon S_{l-1}\to S_l$ is then given by $(x_l, y_l)\mapsto (x_l, x_ly_l)$ or  $(x_l, y_l)\mapsto (x_ly_l, x_l)$ and hence $\pi\colon S\to\p^2$ is of the form $(x_k, y_k)\mapsto (x_k^{r}y_k^{s}, x_k^{t}y_k^{u})$, where $	\left(\begin{array}{cc}
	r & s \\ 
	t & u
	\end{array}  \right) \in\GL_2(\Z)$.
	
	Since $m$ is a monomial transformation, it is of the form $m=(x^ay^b, x^cy^d)$. Hence we obtain that locally $\pi^{-1} m\pi=(x_1^{a'}y_1^{b'},x_1^{c'}y_1^{d'})$, where
	\[
	\left(\begin{array}{cc}
	a' & b' \\ 
	c' & d'
	\end{array}  \right) =\left(\begin{array}{cc}
	r & s \\ 
	t & u
	\end{array}  \right)\left(\begin{array}{cc}
	a & b \\ 
	c & d
	\end{array}  \right)\left(\begin{array}{cc}
	r & s \\ 
	t & u
	\end{array}  \right)^{-1}.
	\]
	Since $\left(\begin{array}{cc}
	a & b \\ 
	c & d
	\end{array}  \right)$ is a loxodromic matrix, the matrix $\left(\begin{array}{cc}
	a' & b' \\ 
	c' & d'
	\end{array}  \right)$ is loxodromic as well. Lemma \ref{loxomonpres} now yields a contradiction to the assumption that $\pi^{-1} m\pi(L)=L$.
\end{proof}

\begin{remark}
	 Let $S\to\p^2$ be a blow-up of toric points. Lemma \ref{monomiallemmanofix} implies in particular that a loxodromic monomial transformation $m$  does not preserve any irreducible curve on $S$, i.e.\,there exists no irreducible curve $C$ such that $m(C)=C$. Indeed, for curves contained in $\partial S$ the claim is proven in Lemma \ref{monomiallemmanofix}. Assume now that there is an irreducible curve $C\subset S$ that is not contained in $\partial S$ that satisfies $m(C)=C$. Let $S'\to S$ be a blow-up of toric points such that $C$ intersects $\partial S'$ in a non-toric point $p$ and let  $L$ be the irreducible boundary component of $\partial S'$  that contains $p$. Since $m$ preserves the complement of $\partial S'$ and $C$ intersects $\partial S'$ in only finitely many points, there exists a positive integer $n$ such that $m^n(p)=p$. But this implies $m(L')=L'$, which contradicts Lemma \ref{monomiallemmanofix}.
\end{remark}

\begin{lemma}\label{goodsetopen}
	Let $K$ be a positive integer and let $f\in\Cr_2(\C)$ be a birational transformation that contracts a curve $C\subset \p^2$ that is not contained in $\partial\p^2$ to a non-toric point $p$ in the bubble space $\mathcal{B}(\p^2)$. Let $U_1, U_2\subset D_2$ be the subsets such that for all $d\in U_1$ and  all $1\leq l\leq K$ we have:
	\begin{itemize}
		\item $(df)^{-l}(C)$ is a curve not contained in $\partial \p^2$;
	\end{itemize}
	and for all $d\in U_2$ and  all $1\leq l\leq K$:
	\begin{itemize}
		\item  $(df)^{l}(C)$ is a non-toric point.
	\end{itemize}
	Then the sets $U_1$ and $U_2$ are open. It follows that for all $d\in U_1\cap U_2$, the transformation $(df)^K$ contracts at least $K$ different irreducible curves.
\end{lemma}

\begin{proof}
	For all $1\leq l\leq K$ the condition that $f^{-1}d_{l}^{-1}f^{-1}\dots f^{-1}d_1^{-1}(C)$ is not contained in the exceptional locus of $f^{-1}$ nor in $\partial\p^2$ is an open condition on $(D_2)^l$. Hence there is an open  set $V_1\subset (D_2)^K$ such that for all $(d_1,\dots, d_k)\in V_1$ and  all $1\leq l\leq K$ we have that $f^{-1}d_{l}^{-1}f^{-1}\dots f^{-1}d_1^{-1}(C)$ is not contained in the exceptional locus of $f^{-1}$ nor in $\partial\p^2$. We embed $D_2$ into $(D_2)^K$ by identifying it with the diagonal. In that way we can define $U_1\coloneqq D_2\cap V_1$. 
	
	To construct $U_2$ we proceed similarly. First we note that there is an open set $V_2\subset(D_2)^K$ such that for all $1\leq l\leq K$ and all $(d_1,\dots, d_k)\in V_2$ the point ${d_l}_\bullet f_\bullet\dots {d_1}_\bullet (f(C))$ is not a base-point of $f$ and is not a point that is mapped to a toric point by $f_\bullet$. We then define $U_2\coloneqq V_2\cap D_2$.
\end{proof}

Define $K\in\Z$ to be the integer from Lemma \ref{proofofloxomain} such that for all loxodromic transformations of monomial type $g$ of degree $\leq \deg(f)$ one has that $g^n$ contracts at most $K-1$ different curves. Assume that the open sets $U_1$ and $U_2$ from Lemma~\ref{goodsetopen} are non-empty. By choosing a $d\in U_1\cap U_2$ such that $df$ is loxodromic, we obtain that $df$ is loxodromic but not of monomial type. The main idea of the proof of Theorem \ref{loxomain} will be to use this kind of argument together with Lemma \ref{sheplemma} to show that loxodromic elements in a simple group $G\subset\Cr_2(\C)$ only contract curves contained in $\partial \p^2$. From this we will then deduce that all loxodromic elements in $G$ are in fact monomial which will lead to a contradiction. However, the cumbersome part is to construct a loxodromic transformation $f$ in $G$ for which the two open sets $U_1$ and $U_2$ are non-empty. 

\begin{lemma}\label{nontoriccontract}
	Let  $f\in\Cr_2(\C)$ be a birational transformation that contracts a curve $C$ that is not contained in $\partial \p^2$ and assume that $f(C)$ is a point not contained in $\partial S$. Let $K\in\Z_+$ be a constant. Then there exists a dense open set $U\subset D_2$ such that $(df)^n(C)$ is a point not contained in $\partial\p^2$ for all $1\leq n\leq K$ and all $d\in U$.
\end{lemma}

\begin{proof}
	By Lemma \ref{goodsetopen} there exists an open set $U\subset D_2$ such that $(df)^n(C)$ is a point not contained in the toric boundary for all $1\leq n\leq K$ and all $d\in U$. It is therefore enough to show that there exists one $d\in D_2$ with this property.
	For this, consider a point $q\in C$ that is not contained in $\partial \p^2$ and is not an indeterminacy point of $f$ and let $d\in D_2$ be the transformation that maps the point $f(C)$ to $q$. It follows that $(df)^n(C)=q$ for all $n\in\Z_+$.
\end{proof}

\begin{lemma}\label{denselemma}
	Let $m\in\Cr_2(\C)$ be a monomial loxodromic transformation, $n\in\Z_+$ and $p\in\p^2$ a point not contained in $\partial\p^2$. The set
	\[
	\{(dm)^n(p)\mid d\in D_2\}
	\]
	contains a dense open set in $\p^2$.
\end{lemma}

\begin{proof}
	By Lemma \ref{puremonomial} we may assume that $m\in\GL_2(\Z)$. We can write $m=(x^ay^b,x^cy^d)$ for some matrix $A\coloneqq\left(\begin{array}{cc}
	a & b \\ 
	c & d
	\end{array}  \right) \in\GL_2(\Z)$. Let $d=(d_1x, d_2y)\in D_2$. One calculates $md=d'm$, where $d'=({d_1}^a{d_2}^bx, {d_1}^c{d_2}^d y)$. Let  $B\coloneqq\left(\begin{array}{cc}
	r & s \\ 
	t & u
	\end{array}  \right)=\id+A+\cdots+A^{n-1}$. Then $(dm)^n=d'm^n$, where $d'=({d_1}^r{d_2}^sx, {d_1}^t{d_2}^u y)$. In order to prove the lemma, we need to show that the morphism $\varphi_B\colon D_2\to D_2$ given by $(d_1x, d_2 y)\mapsto ({d_1}^r{d_2}^sx, {d_1}^t{d_2}^u y)$ is dominant. First note that $(A-\id)B=A^n -\id$. Since $A$ is loxodromic, $A^n$ does not have $1$ as an eigenvalue; hence $\det(A^n-\id)\neq 0$ and therefore $\det(B)\neq 0$. By the Smith normal form, we can write $B=M_1 DM_2$, where $M_1, M_2\in\GL_2(\Z)$ and $D$ is a diagonal integer matrix of rank two as $\det(B)\neq 0$. Since the morphisms from $D_2$ to itself induced by the matrices $M_1, D$ and $M_2$ are all dominant, the morphism $\varphi_B$ is dominant.
\end{proof}

\begin{lemma}\label{lemmacontract}
	Let $f\in\Cr_2(\C)$ be a birational transformation that contracts a curve $C$ that is not contained in $\partial\p^2$.
	Let $m\in\Cr_2(\C)$ be a loxodromic monomial transformation. Then, for every $K\in\Z_+$, there exists a dense open subset $U_K\subset D_2$ such that for all $d\in U_K$ the birational transformation $h_d\coloneqq fdmf^{-1}$ 
	satisfies the following properties:
	\begin{itemize}
		\item The strict transform $\tilde{C}_d\coloneqq (dmf^{-1})^{-1}(C)$ is a curve not contained in $\partial\p^2$, in particular, $h_d$ contracts $\tilde{C}_d$;
		\item $h_d^{-l}(\tilde{C}_d)$ is a curve not contained in $\partial\p^2$ for all $1\leq l\leq K$.
	\end{itemize}
\end{lemma}

\begin{proof}
	Let $p\in C$ be a point that is not contained in $\partial\p^2$. Let $X$ be the union of the exceptional locus of $f$ and the curves that are mapped to $\partial\p^2$ by $f$. Denote by $V_l\subset D_2$ the set of all diagonal automorphisms $d$ such that $(dm)^{-l}(p)$ is not contained in $X$, where $1\leq l\leq K$. By Lemma \ref{denselemma}, the sets $V_l$ contain a subset, which is open and dense in $D_2$. Hence, the intersection $V_1\cap\dots\cap V_{K+1}$ contains a subset $U_K$ that is open and dense in $D_2$. Since, by the construction of $U_K$, the point $(dm)^{-1}(p)$ is not contained in $X$, it follows that the strict transform $\tilde{C}_d=(dmf^{-1})^{-1}(C)$ is indeed a curve and not contained in $\partial \p^2$. This is because $(dm)^{-1}(C)$ is a curve as $dm$ is monomial and $C$ is not contained in $\partial\p^2$. Moreover, the choice of $d$ ensures that $(dm)^{-1}(C)$ is not contained in $X$. Similarly, $(h_d)^{-l}(\tilde{C}_d)=f(dm)^{-l}f^{-1}(\tilde{C}_d)=f(dm)^{-l-1}(C)$ is a curve not contained in $\partial \p^2$ for all $1\leq l\leq K$. 
\end{proof}

\begin{lemma}\label{nontoricdense}
	Let  $f\in\Cr_2(\C)$ a birational transformation that contracts a curve $C$ that is not contained in $\partial \p^2$ and assume that $f(C)$ is a point not contained in $\partial \p^2$. 
	Let $m\in\Cr_2(\C)$ be a monomial loxodromic birational transformation. Then, for every $K\in\Z_+$, there exists a dense open subset $U_K\subset D_2$ and for each $d\in U_K$ there exists a dense open subset $V_K^d\subset D_2$ such that:
	\begin{itemize}
	\item  for all elements $d_1\in U_K$ and for all $d_2\in V_K^{d_1}$ the birational transformation $(d_2fd_1mf^{-1})^K$  is loxodromic and contracts $K$ different irreducible curves. 
	\end{itemize}
\end{lemma}

\begin{proof}
	By Lemma \ref{lemmacontract}, there exists a dense open set $U_K\subset D_2$ such that for all $d_1\in U_K$ the strict transform $\tilde{C}\coloneqq (dmf^{-1})^{-1}(C)$ is a curve not contained in $\partial\p^2$ and such that $(fd_1mf^{-1})^{-l}(\tilde{C})$ is a curve not contained in $\partial \p^2$ for all $1\leq l\leq K$ and all $d\in U_K$. 
	
	Fix now any $d_1\in U_K$. Since, by assumption, $fd_1mf^{-1}(\tilde{C})$ is not contained in $\partial \p^2$, we can apply Lemma \ref{nontoriccontract}. In other words, there exists a dense open subset $V_K^{d_1}\subset D_2$ such that $(d_2fd_1mf^{-1})^l(\tilde{C})$ is a point not contained in the toric boundary for all $d_2\in V_K^{d_1}$ and all $1\leq l\leq K$. This implies in particular, that $(d_2fd_1mf^{-1})^K$ contracts $K$ different curves, namely the curves $(d_2fd_1mf^{-1})^{-l}(\tilde{C})$ for $1\leq l\leq K$. After possibly shrinking $V_K^{d_1}$, we may assume that $d_2fd_1mf^{-1}$ is loxodromic, by Theorem \ref{lowersemi1}.
\end{proof}

\begin{lemma}\label{easycase}
	Let $f\in\Cr_2(\C)$ be a birational transformation that contracts a curve $C$ that is not contained $\partial \p^2$ and assume that $f(C)$ is a point not contained in $\partial \p^2$. 
	Let $m\in\Cr_2(\C)$ be a monomial loxodromic birational transformation and let $\Delta_2\subset D_2$ be a dense subgroup. Then the  group $\langle f, m, \Delta_2 \rangle$ contains a loxodromic element that is not of monomial type.
\end{lemma}

\begin{proof}
	Let $d\coloneqq \deg(f)^2\deg(m)$ and let $K$ be the constant given by Lemma \ref{boundedei} such that all elements in $g\in\Cr_2(\C)$ of monomial type of degree $\leq d$ satisfy the property that the number of irreducible curves $E(g^n)$ contracted by $g^n$ is $<K$ for all $n$. Let $U_K\subset D_2$ and $V_K^{d_1}$ for all $d_1\in U_K$ be the subset given by Lemma \ref{nontoricdense} and fix a $d_1\in U_K\cap\Delta_2$ and $d_2\in V_K^{d_1}\cap \Delta_2$. The birational transformation $d_2fd_1mf^{-1}$ is therefore loxodromic, of degree $\leq d$ and $(d_2fd_1mf^{-1})^K$ contracts $K$ different irreducible curves. It follows that $d_2fd_1mf^{-1}\in \langle f, m, \Delta_2 \rangle$ is loxodromic but not of monomial type.
\end{proof}

\begin{lemma}\label{technicalconj}
	Let $f\in\Cr_2(\C)$ be a birational transformation that contracts a curve $C$ that is not contained in $\partial\p^2$ to a non-toric point in $\mathcal{B}(\p^2)$. Let $m\in\Cr_2(\C)$ be a loxodromic monomial transformation and $\Delta_2\subset D_2$ be a dense subgroup. Then the group $\langle f, m, \Delta_2 \rangle$ contains a loxodromic element $f'$ with the following properties:
	\begin{itemize}
		\item  contracts a curve $\tilde{C}$ not contained in $\partial\p^2$ to a non-toric point in $\mathcal{B}(\p^2)$;
		\item  there exists a dense open subset $U\subset D_2$ such that for all $d\in U$ the transformation $f'd^{-1}f'^{-1}d$ is loxodromic.
	\end{itemize} 
\end{lemma} 

\begin{proof}
	By Lemma \ref{lemmacontract}, there exists for each $K>0$ a dense open set $U_K\subset D_2$ such that $\tilde{C}_d\coloneqq (dmf^{-1})^{-1}(C)$ is a curve not contained in $\partial\p^2$ and  $(fdmf^{-1})^{-l}(\tilde{C}_d)$ is a curve not contained in $\partial\p^2$ for all $1\leq l\leq K$. We fix a $e\in\bigcap_{K\in\Z_+} U_K$. 
	
	By Lemma \ref{loxoconj}, there exists an $n\in\Z_+$ and a dense open subset $U\subset D_2$ such that the transformation $(femf^{-1})^n d^{-1}(femf^{-1})^{-n}d$ is loxodromic for all $d\in U$. By Theorem \ref{lowersemi1}, the subset $V\subset (D_2)^2$ consisting of elements $(d_1, d_2)$ such that $(fd_1mf^{-1})^n d_2^{-1}(fd_1mf^{-1})^{-n}d_2$ is loxodromic, is open and dense. Define the dense open set $V'\coloneqq V\cap (U_n\times D_2)$ and fix $(d_1, d_2)\in V'\cap \Delta_2\times \Delta_2$. We define now $f'\coloneqq (fd_1mf^{-1})^n\in\langle f, m, \Delta_2 \rangle $ and  $\tilde{C}\coloneqq(fdmf^{-1})^{-n+1}(\tilde{C}_{d_1})$. Again by Theorem~\ref{lowersemi1}, there exists a dense open $U\subset D_2$ such that $f'd^{-1}f'^{-1}d$ is loxodromic for all $d\in U$. 
\end{proof}

\begin{lemma}\label{mainlemmalox}
	Let $G\subset\Cr_2(\C)$ be a simple group that contains a loxodromic monomial element $m$. Then $G$ contains no element that contracts a curve that is not contained in the toric boundary $\partial \p^2$ to a non-toric point in $\mathcal{B}(\p^2)$.
\end{lemma}

\begin{proof}
	Since $G$ is simple it contains no tight elements. Hence, by Lemma \ref{sheplemma}, all loxodromic elements in $G$ are of monomial type and $G$ contains a subgroup  $\Delta_2$ that is dense in $D_2$. Assume now that there is an element $f\in G$ that contracts a curve $C\subset\p^2$ that is not contained in $\partial\p^2$ to a non-toric point in $\mathcal{B}(\p^2)$, i.e.\,there exists a blow-up of toric points $\pi\colon S\to \p^2$ such that $f(C)\in S$ is not a toric point (recall that, by abuse of notation, $f$ also denotes the lift of $f$ by $\pi$).
	
	The group $G$ contains no element that contracts a curve that is not contained in $\partial \p^2$ to a  point that is not contained in $\partial \p^2$, by Lemma \ref{easycase}. Hence for all elements $g\in G$ that contract a non-toric curve $D$, the point $g(D)$ is contained in $\partial \p^2$. 
	
	By Lemma \ref{technicalconj}, there exists a loxodromic element $f'\in \langle f, m, \Delta_2\rangle\subset G$ and a curve $\tilde{C}$ in $\p^2$ not contained in $\partial \p^2$ that is contracted by $f'$ to a non-toric point in $\mathcal{B}(\p^2)$, as well as a dense open set $U\subset D_2$ such that for all $d\in U$  the transformation $g_d\coloneqq d^{-1}f'df'^{-1}$ is loxodromic. Moreover, we may choose the dense open set $U\subset D_2$ in such a way that for all $d\in U$ we have that $\tilde{C}_d\coloneqq (df'^{-1})^{-1}(C')=f'd^{-1}(C)$ is a curve  not contained in $\partial \p^2$ and $g_d$ contracts $\tilde{C}_d$  to a non-toric point in $\mathcal{B}(S)$. 
	
	Let $\pi_1\colon S_1\to \p^2$ be a blow-up of toric points and let $L\subset \partial S_1$ be an irreducible component of $\partial S_1$. We claim  that one of the following is true:
	\begin{enumerate}
		\item There exists a dense open subset $U_L\subset U$ such that $g_d(L)$ is not contained in $\partial S_1$ for all $d\in U_L$.
		\item There exists a dense open set $U_L\subset U$ such that $g_d$ contracts $L$ to a non-toric point $p$ in $\mathcal{B}(S_1)$ for all $d\in U_L$. More precisely, there exists a blow-up of toric points $\pi_2\colon S_2\to S_1$ and an irreducible boundary component $L_2\subset\partial S_2$ such that $g_d(L)$ is a proper non-toric point of $L_2$ for all $d\in U_L$ (in particular, $L_2$ does not depend on the choice of $d\in U_L$).
		\item There exists a dense open set $U_L\subset U$ such that for every $d\in U_L\cap\Delta_2$ there is an element $r\in G$ satisfying that $rg_d(L)$ is not contained in $\partial S_1$.
	\end{enumerate}
	Let us now prove the claim. The first observation is that if $f'^{-1}(L)$ is not contained in $\partial S_1$, then there is a dense open set $U_L\subset U$ such that for all $d\in U_L$, the image $d(f'^{-1}(L))$ is neither contained in the indeterminacy locus of $f'$ nor in the set of points that is mapped by $f'$ to $\partial S_1$. This implies that for all $d\in U_L$ the image $g_d(L)=d^{-1}f'df'^{-1}(L)$ is not contained in $\partial S_1$ and we are in situation (1) of our claim. Hence, in what follows we may assume that $(f')^{-1}(L)$ is contained in $\partial S_1$.

	We now distinguish various cases:
	\begin{description}
		\item[Case (a)] Assume that there is a dense open set $U_L\subset U$ such that for all $d\in U_L$, the map $g_d$ does not contract $L$ and that $g_d(L)=L$. Since $G$ is simple, for all $d\in \Delta_2\cap U_L$ the transformation $g_d$ is loxodromic and hence of monomial type. So there exists an $h_d\in\Cr_2(\C)$ and a loxodromic monomial transformation $m_d$ such that $g_d=h_dm_dh_d^{-1}$. By assumption, $g_d(L)=L$, i.e.\,$h_dm_dh_d^{-1}(L)=L$. Note that by Lemma \ref{monomiallemmanofix}, the image $h_d^{-1}(L)$ can not be a curve, since a loxodromic monomial transformation does not preserve any curve on a blow-up of toric points. Hence, either there exists a blow-up of toric points $S_2\to S_1$ such that $h_d^{-1}(L)=L'$ for some irreducible boundary component, or $h_d^{-1}$ contracts $L$ to some non-toric point $p$ in $\mathcal{B}(S_1)$, which has to be a fixed point of $m_d$. The first is not possible, since $m_d(L')\neq L'$ by Lemma \ref{monomiallemmanofix}. In the latter case we conclude that $p$ is not contained in $\partial S_1$, using once more Lemma \ref{monomiallemmanofix}. Since $h_dm_dh_d^{-1}$ is contained in $G$, there exists a dense subgroup $\Delta_2^{m_d}\subset D_2$ such that $h_d\Delta_2^{m_d}h_d^{-1}$ is contained in $G$. Let $W\subset D_2$ be the dense open subset such that $d(p)$ is not a base-point of $h_d$ and is not contained in the set of points that is mapped to $\partial S_1$ by $h_d$. For a $c\in\Delta_2^{m_d}\cap W$ and $r\coloneqq h_dch_d^{-1}$ we have $rg_d= h_dcm_dh_d^{-1}$, and therefore $rg_d(L)=h_dcm_dh_d^{-1}(L)$ is not contained in $\partial S_1$. Hence we are in situation (3).
		\item[Case (b)] Assume that $f'^{-1}$ contracts $L$ to a non-toric point in $\mathcal{B}(S_1)$, i.e.\, there exists a blow-up of toric points $S_1'\to S_1$ and a smooth rational curve $L_1'\subset\partial S_1'$ such that $f'^{-1}(L)$ is a non-toric point on $L_1'$. Let $U_L\subset U$ be the dense open set such that $df'^{-1}(L)$ is neither contained in the indeterminacy locus of $f'$ nor in the set of points of $L_1'$ that are mapped by $f'_\bullet$ to a toric point in $\mathcal{B}(S_1')$ for all $d\in U_L$. We are then in situation (2). Moreover, $L_2$ does not depend on the choice of $d\in U_L$.
		\item[Case (c)] Assume that $f'^{-1}(L)=L_1$, where $L_1\subset\partial S_1$ is an irreducible boundary component. In this case, $g_d(L)=L$ for all $d\in U$ and we are in case (a).
		\item[Case (d)] Assume that $f'^{-1}$ contracts $L$ to a toric point in $\mathcal{B}(S_1)$. In this case, there is a blow-up of toric points $S_2\to S_1$ such that $f'^{-1}(L)=L_2$, where $L_2\subset\partial S_2$ is an irreducible boundary component. This reduces to case (c).
		
	\end{description}
	This proves the claim.

	Let $d\in U$ and consider for each $e\in D_2$ the loxodromic transformation $g_d emg_d^{-1}$. The transformation $d_1(g_d emg_d^{-1})$ is loxodromic for all $d_1$ in a dense open subset of $D_2$, by Theorem \ref{lowersemi}. The degree of $d_1(g_d emg_d^{-1})$ is at most $\deg(f')^4\deg(m)$ for all $d_1, d, e\in D_2$. By Lemma \ref{boundedei} there exists a constant $K\in\Z$, such that if $d_1(g_d emg_d^{-1})$ is of monomial type, then $(d_1(g_d emg_d^{-1}))^n$ contracts at most $K-1$ different curves for all $n\in\Z$.

	There exists for every $d\in U_1\coloneqq U$ a dense open subset $V^1_d\subset D_2$ such that $\tilde{C}_{e,d}'\coloneqq (emg_d^{-1})^{-1}(\tilde{C}_d)$ is a curve for all $e\in V_1^d$ (recall that $\tilde{C}_d$ is a curve not contained in $\partial\p^2$ that is contracted by $g_d$). Moreover, for all $d\in U_1$ and all $e\in V_1^d$ there exists a dense open subset $W_{d}^1$ of $D_2$ such that $d_1(g_d emg_d^{-1})(\tilde{C}_{e,d}')$ is not a base-point of $g_d^{-1}$ (such a dense open set exists, since $(g_d emg_d^{-1})(\tilde{C}_{e,d}')$ is not a toric point). We will now inductively add additional open conditions on the sets $U^1$, $V^1_d$ and $W_{d}^1$. If $p_1\coloneqq d_1(g_d emg_d^{-1})(\tilde{C}_{e,d}')$ is not contained in the toric boundary for some $d_1, d, e\in D_2$, then $p_1$ is not contained in the toric boundary for all $d_1, d, e$ in a dense open subset $T\subset (D_2)^3$. By chosing $(d_1, d, e)\in T\cap (\Delta_2)^3$ we obtain an element in $G$ that contracts the curve $\tilde{C}_{e,d}'$, which is not contained in the toric boundary, to a point outside the toric boundary, which is not possible, by Lemma~\ref{easycase}. Hence, $p_1$ is contained in $\partial \p^2$. After a blow-up of toric points $S_1\to \p^2$ we may assume that $p_1$ is a proper non-toric point of $\partial S_1$, which is, by our condition on $d_1$, not a base-point of $g_d^{-1}$. Let $L_1\subset\partial S_1$ be the line containing $p_1$, hence we are in one of the situations (1) to (3) described above. The situations (1) and (3) do not occur since otherwise we would obtain an element in $G$ that contracts a curve not contained in the toric boundary to a point outside the toric boundary. Hence we are in situation (2), i.e.\,there exists a dense open set $U_{L_1}\subset D_2$ such that $g_d^{-1}$ contracts $L_1$ to a non-toric point in $\mathcal{B}(S_1)$. We set $U_2\coloneqq U_{L_1}\cap U_1$. Let $S_2\to S_1$ be the blow-up of toric points such that $p_2\coloneqq g_d^{-1}(p_1)=g_d^{-1}(L_1)$ is a proper non-toric point on a line $L_2\subset \partial S_2$. The crucial point here is that $p_2$ does not depend on the choice of $d_1$ nor on the choice of $e$. After a blow-up of toric points $S_2'\to S_2$, the monomial map $m$ maps $L_2$ to another component $L_2'\subset\partial S_2'$. We define for all $d\in U_2$ the set $V_d^2$ as the set of all $e\in V_d^1$ such that $em(p_2)$ is not an indeterminacy point of $g_d$. Again, by the same argument as above, we obtain that $g_d$ contracts $L_2$ to a non-toric point in $\mathcal{B}(S_2)$ that lies on or above the toric bundary. After a blow-up of toric points $S_3\to S_2$ we may assume that $p_3\coloneqq g_d(L_2)=g_dem(p_2)=g_demg_d^{-1}(d_1g_demg_d^{-1})(\tilde{C}_{e,d}')$ is a proper non-toric point on a line $L_3\subset\partial S_3$. And again, $p_3$ does not depend on the choice of $d_1$ and $e$, so we obtain an additional open condition on the choice of $d_1$ and thus a dense open subset $W_d^2\subset W_d^1$. We now continue this process and obtain a sequence of blow-ups of toric points $S_{2k-1}\to\cdots\to S_2\to S_1$ and a sequence of lines $L_1, L_2,\dots, L_{2K}$, where $L_i\subset \partial S_i$, as well as inclusions of dense open sets $U_{2K}\subset U_{2K-1}\subset\cdots\subset U_1$ and, for every $d\in U_{2K}$, inclusions of dense open sets $W_d^K\subset\cdots \subset W_d^2\subset W_d^1$ and $V_d^K\subset\cdots \subset V_d^2\subset V_d^1$ with the property that for all $d\in U^{2K}$ and all $d_1\in W_d^K$, $e\in V_d^K$, one has that $(d_1g_demg_d^{-1})^l(\tilde{C}_{e,d}')$ is a point on the line $L_{2l-1}$ for all $1\leq l\leq K$ and that $g_d^{-1}(d_1g_demg_d^{-1})^l(\tilde{C}_{e,d}')$ is a point on the line $L_{2l}$. In particular, $(d_1g_demg_d^{-1})^l(\tilde{C}_{e,d}')$ is a point for all $1\leq l\leq K$. 
	
	We fix an element $d\in U_{2K}\cap\Delta_2$. By Lemma \ref{lemmacontract}, there exists a dense open set $V\subset D_2$, such that $(g_demg_d^{-1})^{-l}(\tilde{C}_{e,d}')$ is a curve not contained in $\partial \p^2$ for all $1\leq l\leq K$ and all $e\in V$. We fix an element $e\in V\cap V_d^K\cap\Delta_2$. By Lemma \ref{goodsetopen}, there exists a dense open set $W\subset D_2$ such that $d_1(g_demg_d^{-1})$ is loxodromic for all $d_1\in W$ and a dense open set $W'\subset D_2$ such that $(d_1g_demg_d^{-1})^{-l}(\tilde{C}_{e,d}')$ is a curve not contained in $\partial \p^2$. Fix a $d_1\in W\cap W'\cap W_d^K\cap \Delta_2$. Then the transformation $h\coloneqq d_1g_demg_d^{-1}$ has the following properties:
	\begin{enumerate}
		\item $h$ is contained in $G$;
		\item $h$ is loxodromic;
		\item $h^l(\tilde{C}_{e,d}')$ is a point for all $1\leq l\leq K$;
		\item $h^{-l}(\tilde{C}_{e,d}')$ is a curve not contained in $\partial \p^2$ for all $1\leq l\leq K$.
	\end{enumerate}
	The properties (3) and (4) imply that $h^K$ contracts $K$ different curves. Hence, by definition of $K$, $h$ is not of monomial type. But this is a contradiction to $G$ being simple, as was explained before.
\end{proof}

\begin{lemma}\label{onlytoric}
	Let $f\in\Cr_2(\C)$ be a loxodromic element with the following property: 
	\begin{itemize}
		\item For all $n\in\Z$ no non-toric curve $C$ is contracted by $f^n$ to a non-toric point in the bubble space $\mathcal{B}(\p^2)$.
	\end{itemize}
	Then either $f$ is monomial or there exists a dense open subset $U\subset D_2$ and an $n\in\Z$ such that $df^nd^{-1}f^{-n}$ is loxodromic and not of monomial type.
\end{lemma}

\begin{proof}
	Assume that $f$ is not monomial. Then $f^n$ is not monomial for all $n\neq 0$, by Lemma \ref{loxogllemma}.
	By Lemma \ref{loxoconj} there exists a dense open set $U\subset D_2$ and an $m\in\Z$ such that $df^md^{-1}f^{-m}$ is loxodromic for all $d\in U$. 
	
	All curves that are not contained in the toric boundary and that are contracted by $f^m$ are contracted to toric points in $\mathcal{B}(\p^2)$ and these are fixed by diagonal automorphisms. Hence, for all $d\in U$ the map $(f^mdf^{-m})^n=f^md^nf^{-m}$ contracts only toric curves for all $n\in\Z$. Denote by $B\subset\partial\p^2$ the union of all the coordinate lines that are contracted by $f^md^nf^{-m}$ for some $n\in\Z$. We observe that $f^md^nf^{-m}$ is an isomorphism on $\p^2\setminus B$. As $df^md^{-1}f^{-m}$ is loxodromic, the map $f^mdf^{-m}$ can not be an automorphism of $\p^2$. If $B$ consists of one line, then $f^mdf^{-m}$ and $d^{-1}f^mdf^{-m}$ are automorphisms of $\A^2$. Since the dynamical degree of an element in $\aut(\A^2)$ is always an integer, it follows that $d^{-1}f^mdf^{-m}$ is not of monomial type. If $B$ is the union of two coordinate lines, then $f^mdf^{-m}$ and $d^{-1}f^mdf^{-m}$ are automorphisms of $\A^1\times\A_*^1$. The $\A^1$-fibration of $\A^1\times\A_{*}^1$ is given by the invertible functions on $\A^1\times\A^1_*$, so automorphisms of  $\A^1\times\A^1_*$ preserve this $\A^1$-fibration. In particular, $\A^1\times\A^1_*$ does not admit any loxodromic automorphism which implies that this case does not occur. Finally, if $B$ is the union of all the three coordinate lines, then $fdf^{-1}$ is an automorphism of $\A^1_*\times\A_*^1$, i.e.\,a monomial map. By Lemma \ref{loxogllemma}, the transformation $f$ is monomial.
\end{proof}

\subsection{Proof of Theorem \ref{loxomain}}\label{proofofloxomain} Let $G\subset\Cr_2(\C)$ be a simple subgroup and assume that $G$ contains loxodromic elements. By Lemma \ref{sheplemma}, all loxodromic elements are  of monomial type. Assume that $G$ contains a loxodromic element $m$. After conjugation we may therefore assume that $m$ is monomial. From Lemma \ref{mainlemmalox} it follows that all the curves contracted by elements of $G$ are toric, and hence Lemma~\ref{onlytoric} implies that all loxodromic elements of $G$ are contained in $\GL_2(\Z)\ltimes D_2$. Let $h\in G$ be an arbitrary element. Since $hgh^{-1}$ is loxodromic, it is monomial. By Lemma \ref{loxogllemma}, $h$ is contained in $\GL_2(\Z)\ltimes D_2$ as well. Therefore $G\subset \GL_2(\Z)\ltimes D_2$ and we obtain a non-trivial homomorphism $\varphi\colon G\to\GL_2(\Z)$ whose kernel contains $\Delta_2$ - a contradiction to $G$ being simple. Therefore, $G$ does not contain any loxodromic element.
\qed

\section{Proof of Theorem \ref{mainthm} and Theorem \ref{mainsimple}}

\subsection{The parabolic and elliptic case}

\begin{lemma}\label{paracase}
	let $G\subset\Cr_2(\C)$ be a simple subgroup that contains no loxodromic element, but a parabolic element. Then $G$ is conjugate to a subgroup of the de Jonqui\`eres group and $G$ is isomorphic to a subgroup of $\PGL_2(\C)$.
\end{lemma}

\begin{proof}
	By Lemma \ref{noloxfibration}, we know that $G$ is either conjugate to a subgroup of the automorphism group of a Halphen surface or to a subgroup of the de Jonqui\`eres subgroup $\J$. By Theorem \ref{halphenstructure}, automorphism groups of Halphen surfaces are finite extensions of abelian subgroups. It follows that the automorphism group of a Halphen surface does not contain infinite simple subgroups. Therefore, $G$ is conjugate to a subgroup of $\J$. Let 
	\[
	1\to\PGL_2(\C(t))\to\J\to\PGL_2(\C)\to 1
	\]
	be the short exact sequence from the semi-direct product structure of $\J$. Since $G$ is simple, it is either contained in the kernel or the image of $\varphi$. In both cases it is isomorphic to a subgroup of $\PGL_2(\C)$.
\end{proof}

\begin{lemma}\label{ellcase}
	Let $G\subset\Cr_2(\C)$ be a simple subgroup of elliptic elements. Then either $G$ is a subgroup of an algebraic group in $\Cr_2(\C)$ or $G$ is conjugate to a subgroup of the de Jonqui\`eres group $\J$. 
\end{lemma}

\begin{proof}
	Let $G\subset\Cr_2(\C)$ be a simple subgroup of elliptic elements. If $G$ is a subgroup of an algebraic group or if $G$ is conjugate to a subgroup of the de Jonqui\`eres group, we are done. So by Theorem \ref{ellipticelements}, it only remains to consider the case where $G$ is a torsion group. In this case, $G$ is isomorphic to a subgroup of an algebraic group, by Theorem \ref{ellipticelements}, and as such it is a linear group. The Theorem of Jordan and Schur implies that $G$ has a normal abelian subgroup of finite index. This implies that $G$ is finite and therefore algebraic.
\end{proof}

\subsection{Proofs} We have now gathered all the results to prove Theorem \ref{mainsimple} and Theorem~\ref{mainthm}:

\begin{proof}[Proof of Theorem \ref{mainsimple}]
	The first statement of the Theorem is proven in Theorem~\ref{loxomain}, the second statement of Theorem~\ref{mainsimple} is proven in Lemma~\ref{paracase} and the third statement in Lemma~\ref{ellcase}.
\end{proof}

\begin{proof}[Proof of Theorem \ref{mainthm}]
	Let $G$ be a simple group acting non-trivially on a complex rational surface $S$. 
	If $S$ is rational it follows from the classification of maximal algebraic groups (Theorem \ref{maxsubgroups}) and Theorem \ref{mainsimple} that $G$ is isomorphic to a subgroup of $\PGL_3(\C)$. 
	If $S$ is  non-rational the proof follows from Lemma \ref{negativekodaira} and Lemma~\ref{nonnegativekodaira}.
	
	On the other hand, if $S$ is rational, then $\PGL_3(\C)=\aut(\p^2)$ is a subgroup of $\Bir(S)$, and in particular, every simple subgroup of $\PGL_3(\C)$ acts by birational transformations on $S$. If $G$ is isomorphic to a subgroup of $\PGL_2(\C)$, then it acts non-trivially by birational transformations on the surface $\p^1\times C$ for all curves $C$. For every finite group $G$ there exists a curve of general type $C$ such that $\aut(C)=G$. Hence, $G$ acts non-trivially by birational transformations on the non-rational surface of negative Kodaira dimension $\p^1\times C$ as well as on the surface of non-negative Kodaira dimension $C\times C$.  
\end{proof}

\section{Finitely generated subgroups}
In this section we prove Theorem \ref{fgsimple}.
The main advantage when working with a finitely generated group $\Gamma$, is that we can reduce modulo $p$ the coefficients needed to define the elements in $\Gamma$. We start by  explaining this construction and will then apply it in a second step to our problem.
The following lemma is well known. A proof can be found for example in \cite[Lemma 3.2]{nica2013linear}:

\begin{lemma}\label{nicalemma}
	Let $A$ be a finitely generated domain. The intersection of all maximal ideals of $A$ is $0$. Moreover, if $A$ is a field, then $A$ is finite.
\end{lemma}

The following proposition shows how Lemma \ref{nicalemma} can be applied to obtain information about the structure of subgroups of $\Cr_2(\kk)$ for any field $\kk$. A similar statement has already been proved and applied by de Cornulier in order to show that the Cremona group is sofic (\cite{MR3160544}, see also \cite{cantatremark}). 

\begin{proposition}\label{redmodp}
	Let $\kk$ be a field and let $\Gamma\subset\Cr_n(\kk)$ be a finitely generated subgroup. Then there exists a finite field $\mathbb{F}$ and a non-trivial group-homomorphism $\varphi\colon\Gamma\to\Cr_n(\mathbb{F})$ that satisfies $\deg(\varphi(f))\leq\deg(f)$ for each $f\in\Gamma$.
\end{proposition}

\begin{proof}
	Let $g_1,\dots, g_l\in\Gamma$ be a symmetric set of generators. We may assume that $g_i\neq\id$ for all $i$. Fix homogeneous polynomials $G_{ij}\in\C[x_0, \dots, x_n]$ such that $g_i=[G_{i0}:\dots:G_{in}]$ and define the endomorphisms  $G_i\coloneqq (G_{i0},\dots, G_{in})\in\edo(\A^{n+1})$. Assume that $g_i^{-1}=g_j$ and let
	\[
	F_{i}\coloneqq G_i\circ G_j =(F_{i0},\dots, F_{in})\in\A^{n+1}.
	\]
	Note that $g_i\circ g_j=[F_{i0}:\dots:F_{in}]=[x_0:\dots:x_n]$, i.e.\,$F_{ij}=P_ix_j$ for some homogeneous polynomial $P_i\in \C[x_0,\dots, x_n]$. 
	
	Let $T$ be the finite set of all non-zero coefficients that appear in the polynomials $G_{ij}$, the $F_{ij}$ or the polynomials $G_{1i}G_{2j}-G_{1j}G_{2i}$ and denote by $r$ the product of all elements of $T$. Let $A$ be the domain generated by the elements of $T$ and by $1$.  In particular, we may consider all our polynomials $G_{ij}$, $F_{ij}$ and $G_{1i}G_{2j}-G_{1j}G_{2i}$ to be elements of $A[x_0,\dots,x_n]$. By Lemma~\ref{nicalemma}, there exists a maximal ideal $I\subset A$ such that $r\notin I$. Reduction modulo $I$ yields a homomorphism $\pi\colon A\to\F$ for some finite field $\mathbb{F}$ such that $r$ and hence all elements in $T$ are not contained in the kernel of $\pi$. By reducing the coefficients modulo $I$, we obtain a ring homomorphism $\psi\colon A[x_0,\dots,x_n]\to \F[x_0,\dots,x_n]$. Define the rational maps
	\[
	\varphi(g_i)=[\psi({G_{i0}}):\dots:\psi({G_{in}})].
	\]
	Note that \[
	\varphi(g_i)\circ\varphi(g_i^{-1})=[\psi(F_{i1}):\dots:\psi({F_{in}})]=[\psi(P_i)x_0:\dots:\psi(P_i)x_n]=\id,
	\]
	so $\varphi(g_i)$ is a birational transformation of $\p^n_{\F}$. Assume that $g_{i_1}g_{i_2}\cdots g_{i_l}=\id$ for some $1\leq i_1,\dots, i_l\leq k$. Then $G_{i_1}\circ\cdots\circ G_{i_l}=(Qx_0,\dots, Qx_n)$ for some homogeneous polynomial $Q$. It follows that $\varphi(g_{i_1})\varphi(g_{i_2})\cdots \varphi(g_{i_l})=[\psi(Q)x_0:\psi(Q)x_1:\dots:\psi(Q)x_n]=\id$. Therefore, the map $\varphi$ can be extended to a homomorphism of groups $\varphi\colon\Gamma\to\Cr_n(\F)$. By construction, at least one of the polynomials $\psi(G_{1i})\psi(G_{2j})-\psi(G_{1j})\psi(G_{2i})$ is not zero and hence $\varphi(g_1)\neq\varphi(g_2)$; in particular, $\varphi$ is not trivial. 
	
	Let $g=g_{i_1}g_{i_2}\cdots g_{i_l}\in\Gamma$. Then $g=[H_0:H_1:\dots:H_n]$, where $(H_0,\dots, H_n)=G_{i_1}\circ\cdots\circ G_{i_l}$. We then have $\varphi(g)=[\psi(H_0):\psi(H_1):\dots:\psi(H_n)]$. This shows that $\deg(\varphi(g))\leq\deg(g)$.
\end{proof}

Together with Theorem \ref{sheprel} we obtain the following result:

\begin{proposition}\label{mainfg}
	Let $\kk$ be an algebraically closed field and let $\Gamma\subset\Cr_2(\kk)$ be a finitely generated subgroup. If $\Gamma$ contains a loxodromic element, then $\Gamma$ is not simple.
\end{proposition}

\begin{proof}
	Let $f\in\Gamma$ be a loxodromic element. If there exists a $n$ such that $f^n$ is tight in $\Gamma$, the group $\Gamma$ is not simple by Theorem \ref{smallcanc} and we are done. If no power of $f$ is tight, it follows from Theorem \ref{sheprel} that $\Gamma$ contains an infinite subgroup $\Delta_2$ that is normalized by $f$ and that is conjugate either  to a subgroup of $D_2$ or to a subgroup of $\kk^2$. The group $\Delta_2$ being conjugate to a subgroup of $D_2$ or $\kk^2$ implies in particular that the degrees of the elements in $\Delta_2$ are uniformly bounded by an integer $K$. By Proposition \ref{redmodp}, there exists a finite field $\F$ and a non-trivial group homomorphism $\varphi\colon\Gamma\to\Cr_2(\mathbb{F})$ that satisfies $\deg(\varphi(f))\leq\deg(f)$. In $\Cr_2(\mathbb{F})$ there exist only finitely many elements of degree $\leq K$, hence the image $\varphi(\Delta_2)$ is finite. It follows that $\varphi$ has a proper kernel and therefore that $\Gamma$ is not simple.
\end{proof}

We are now able to prove Theorem \ref{fgsimple} using the same strategy as in the proof of Theorem \ref{mainsimple}.

\begin{lemma}\label{fgjonq}
	Let $C$ be a curve over an algebraically closed field $\kk$ and $\Gamma\subset\Bir(\p^1\times C)$ be a finitely generated simple subgroup that preserves the $\p^1$-fibration given by the projection to $C$. Then $\Gamma$ is finite.
\end{lemma}

\begin{proof}
	Since $\Gamma$ is simple, it is either isomorphic to a subgroup of $\PGL_2(\kk(C))$ or to a subgroup of $\aut(C)$. Since both, $\PGL_3(\kk(C))$ and $\aut(C)$  satisfy the property of Malcev by (\cite[Corollary 1.2]{MR693651}), the group $\Gamma$ is finite.  
\end{proof}

\begin{proof}[Proof of Theorem \ref{fgsimple}]
	Let $\overline{\kk}$ be the algebraic closure of $k$. Since $\Bir(S_{\kk})\subset \Bir(S_{\overline{\kk}})$, it is enough to show the statement for algebraically closed fields. 
	
	First assume that our surface $S$ is rational. By Proposition \ref{mainfg}, $\Gamma$ does not contain any loxodromic element. If $\Gamma$ contains a parabolic element, then $\Gamma$ is conjugate to a subgroup of the de Jonqui\`eres group $\J\simeq\PGL_2(\C(t))\rtimes\PGL_2(\C)$ or to a subgroup of the automorphism group $\aut(X)$ of a Halphen surface $X$. This last case is not possible by the property of Malcev for automorphism groups (\cite[Corollary 1.2]{MR693651}). If $\Gamma$ is a subgroup of $\J$, the claim follows with Lemma \ref{fgjonq}. If all elements in $\Gamma$ are elliptic, the claim follows from Theorem \ref{cantatelliptic}. In the first case, $\Gamma$ is finite by Lemma \ref{fgjonq}. As for the second case we recall that algebraic subgroups of $\Cr_2(\kk)$ are always linear. Hence $\Gamma$ is linear and therefore finite, since linear groups satisfy the property of Malcev.
	
	If $S$ is a non-rational ruled surface, the statement follows from Lemma \ref{fgjonq}. If $S$ is non-rational and not ruled, it has a unique minimal model $S'$ (see \cite[Corollary~10.22]{MR1805816}). Hence $\Gamma$ is conjugate to a subgroup of $\Bir(S')=\aut(S')$ and is therefore finite by the property of Malcev for automorphism groups (\cite[Corollary~1.2]{MR693651}). 
\end{proof}

\bibliographystyle{amsalpha}
\bibliography{/Users/christianurech/Dropbox/Literatur/bibliography_cu}

\end{document}